\input ./style/arxiv-general.cfg
\documentclass[aop,MSNbibl,dvips]{arximspdf}
\makeatletter
   \@ifpackageloaded{graphicx}{}{\usepackage{graphicx}}
\makeatother

\doi{10.1214/13-AOP867} 
\volume{43}
\issue{3}
\pubyear{2015}
\firstpage{992}
\lastpage{1044}

\makeatletter

\newcommand{\rrvert}{\vert}
\newcommand{\llvert}{\vert}
\newcommand{\eqref}[1]{(\ref{#1})}
\newtheorem{theorem}{Theorem}
\newtheorem{corollary}[theorem]{Corollary}
\newtheorem{lemma}[theorem]{Lemma}
\newproclaim{example}[theorem]{Example}
\newproclaim{remark}[theorem]{Remark}

\newcommand{\rd}{\mathbb{R}^d}
\newcommand{\e}{\mathbf {E}}

\makeatother

\begin{document}
\begin{frontmatter}

\title{Random walks in cones\thanksref{T1}}
\runtitle{Random walks in cones}
\thankstext{T1}{Supported by DFG.}

\begin{aug}
\author[A]{\fnms{Denis} \snm{Denisov}\corref{}\ead[label=e1]{denis.denisov@manchester.ac.uk}}
\and
\author[B]{\fnms{Vitali} \snm{Wachtel}\ead[label=e2]{wachtel@mathematik.uni-muenchen.de}}

\runauthor{D. Denisov and V. Wachtel}

\affiliation{University of Manchester and University of Munich}

\address[A]{School of Mathematics\\
University of Manchester\\
Oxford Road\\
Manchester M13 9PL\\
United Kingdom\\
\printead{e1}}

\address[B]{Mathematical Institute\\
University of Munich\\
Theresienstrasse 39\\
D--80333 Munich\\
Germany\\
\printead{e2}}
\end{aug}

\received{\smonth{11} \syear{2012}}
\revised{\smonth{6} \syear{2013}}

%
\begin{abstract}
We study the asymptotic behavior of a multidimensional random walk
in a general cone. We find the tail asymptotics for the exit time
and prove integral and local limit theorems for a random walk
conditioned to stay in a cone. The main step in the proof consists
in constructing a positive harmonic function for our random walk
under minimal moment restrictions on the increments. For the proof
of tail asymptotics and integral limit theorems, we use a strong
approximation of random walks by Brownian motion. For the proof
of local limit theorems, we suggest a rather simple approach, which
combines integral theorems for random walks in cones with
classical local theorems for unrestricted random walks. We also discuss
some possible applications of our results to ordered random walks and
lattice path enumeration.
\end{abstract}

%


\end{frontmatter}

\section{Introduction, main results and discussion}

\subsection{Motivation}
Random walks conditioned to stay in cones is a very popular topic in
probability. They appear naturally in many situations. Here, we mention
some of them:
\begin{itemize}
\item Nonintersecting paths, which can be seen as a multidimensional
random walk in one of Weyl chambers, are used in modeling of
different physical phenomena; see Fisher \cite{Fisch84}. There are
also a lot of connections between nonintersecting paths and Young
diagrams, domino tiling, random matrices and many other
mathematical objects; for an overview, see K\"onig~\cite{K05}.
\item Random walks in the quarter-plane reflected at the boundary
are often used in the queueing
theory. For diverse examples, see monographs written by Cohen~\cite{Cohen92},
by Fayolle, Iasnogorodski and Malyshev \cite{FIM}
and a paper by Greenwood and Shaked \cite{GS77}.
\item Asymptotic behavior of branching processes and random walks in
random environment is closely connected to the behavior of random
walks conditioned to stay positive, which are one-dimensional cases
of a random walk conditioned to stay in a cone; see \cite{AGKV} and
references therein.
\end{itemize}

The main purpose of the present paper is to propose an approach which
determines the asymptotic behavior of exit times and allows one to
prove limit theorems for a rather wide class of cones and under
minimal moment conditions on the increments of random walks.
For that, we use a strong approximation of multidimensional random
walks with
multidimensional Brownian motion. This allows to extend the
corresponding results
for the Brownian motion to the discrete time setting and to study the
asymptotic behavior
of random walks.

For Brownian motion, the study of exit times from cones
was initiated by Burkholder. In \cite{Burk77},
he proposed necessary and sufficient conditions for the existence
of moments of exit times. Later on, using these results,
DeBlassie \cite{DeB87} found an exact formula for $\mathbf P(\tau_x>n)$
as an infinite series. This formula allowed him to obtain
tail asymptotics for exit times.
These results were obtained by Ba{\~n}uelos and Smits \cite{BS97} under
more general conditions.
Garbit \cite{Gar09} defined a Brownian motion started at origin and
conditioned to stay in
a cone.

For random walks in discrete time, much less is known.
A corresponding
generalization of Burkholder's results was obtained by McConnell
\cite{MC84}. Namely, he found necessary and sufficient conditions
for the existence of moments of exit times.
Varopoulos \cite{Var99,Var00} derived upper and lower
bounds for the tail probability under an additional
assumption that the increments of the random walk are bounded.
Moreover, he showed that this upper bound remains valid for
Markov chains with zero drifts and bounded increments. MacPhee, Menshikov
and Wade~\cite{MMW} gave criteria for the existence of moments of
exit times from wedges for Markov chains with asymptotically zero drifts
and bounded increments.
The exact asymptotic behavior for the exit times of
a random walk is known only in some special
cases. Shimura \cite{Shim84} and Garbit \cite{Gar10} obtained the
asymptotics of the tail and some limit theorems for two-dimensional
random walks. There are many results in the literature on random walks
in Weyl chambers. We shall mention them later, in a special paragraph
devoted to Weyl chambers.

\subsection{Notation and assumptions}
Consider a random walk $\{S(n),n\geq1\}$ on $\rd$, $d\geq1$, where
\[
S(n)=X(1)+\cdots+X(n) %
\]
and $\{X(n), n\geq1\}$ is a family of independent copies of a random
variable $X=(X_1,X_2,\ldots,X_d)$. Denote by $\mathbb{S}^{d-1}$ the
unit sphere of $\rd$ and $\Sigma$ an open and connected subset of
$\mathbb{S}^{d-1}$. Let $K$ be the cone generated by the rays
emanating from the origin and passing through $\Sigma$, that is,
$\Sigma=K\cap\mathbb{S}^{d-1}$.

Let $\tau_x$ be the exit time from $K$ of the random walk with
starting point $x\in K$, that is,
\[
\tau_x=\inf\bigl\{n\ge1\dvtx x+S(n)\notin K\bigr\}. %
\]
In this paper, we study asymptotics for
\[
\mathbf P(\tau_x>n),\qquad n\to\infty, %
\]
construct a random walk conditioned to stay in the cone $K$ and
prove limit theorems for this random walk.

The essential part of the proof is a coupling with the Brownian
motion. Hence, we extensively use related results for the Brownian
motion. Let $B(t)$ be a standard Brownian motion on $\rd$
and let $\tau^{\mathrm{bm}}_x$ be the exit time of $B(t)$ from the cone $K$,
\[
\tau^{\mathrm{bm}}_x=\inf\bigl\{t\ge0\dvtx x+B(t)\notin K\bigr\}.
\]
The harmonic function of the Brownian motion killed at the
boundary of $K$ can be described as the minimal (up to a
constant), strictly positive on $K$ solution of the
following boundary problem:
\[
\Delta u(x)=0, \qquad x \in K \mbox{ with boundary condition }u |_{\partial
K}=0.
\]
If such a function exists, then
see \cite{DeB87} and \cite{BS97}. One can show that
%
\begin{equation}
\label{BS} \mathbf P\bigl(\tau_x^{\mathrm{bm}}>t\bigr)\sim
\varkappa\frac{u(x)}{t^{p/2}},\qquad t\to \infty.
\end{equation}

The function $u(x)$ and constant $p$ can be found as follows.
If $d=1$, then we have only one nontrivial cone $K=(0,\infty)$.
In this case, $u(x)=x$ and $p=1$. Assume now that $d\geq2$.
Let $L_{\mathbb{S}^{d-1}}$ be the Laplace--Beltrami operator on
$\mathbb{S}^{d-1}$ and assume that $\Sigma$ is regular with respect to
$L_{\mathbb{S}^{d-1}}$.
With this assumption, there exists a complete set of orthonormal eigenfunctions
$m_j$ and corresponding eigenvalues $0<\lambda_1<\lambda_2\le\lambda
_3\le\cdots$ satisfying
%
\begin{eqnarray}
\label{eq.eigen} L_{\mathbb{S}^{d-1}}m_j(x)&=&-\lambda_jm_j(x),\qquad
x\in\Sigma,
\nonumber
\\[-8pt]
\\[-8pt]
\nonumber
m_j(x)&=&0, \qquad x\in\partial\Sigma.
\end{eqnarray}
Then
\[
p=\sqrt{\lambda_1+(d/2-1)^2}-(d/2-1)>0
\]
and the harmonic function $u(x)$ of the Brownian motion is given by
%
\begin{equation}
\label{u.from.m} u(x)=|x|^pm_1 \biggl(\frac{x}{|x|}
\biggr),\qquad x\in K.
\end{equation}
We refer to \cite{BS97} for further details on exit times of Brownian motion.

Unfortunately, we are not able to determine the asymptotic behavior
of exit times for random walks for such a general class of cones.
More precisely, we will use the following additional conditions on the
cone $K$:
\begin{itemize}
\item We assume that there exists an open and connected set
$\widetilde\Sigma\subset\mathbb{S}^{d-1}$ with $\operatorname{
dist}(\partial\Sigma, \partial\widetilde\Sigma)>0$ such that
$\Sigma\subset\widetilde\Sigma$ and the
function $m_1$ can be extended to $\widetilde\Sigma$ as a solution to
(\ref{eq.eigen}).
\item$K$ is either convex or starlike [there exists $x_0\in\Sigma$
such that
$x_0+K\subset K$ and $\operatorname{ dist}(x_0+K, \partial K)>0$] and $C^2$.
(Every convex cone is also starlike, for the proof see Remark~\ref{conv=star}.)
\end{itemize}
It is known that if $m_1$ can be extended then the boundary
$\partial\Sigma$ should be piecewise real-analytic. Furthermore, if
$\partial\Sigma$ is real-analytic, then $m_1$ is extendable; see, for
example, Theorem A in
Morrey and Nirenberg \cite{MN}.\hskip.2pt\setcounter{footnote}{1}\footnote{We are grateful to Professor
Ancona for pointing out the reference.}
Since the boundary of every two-dimensional cone consists of two
points on the unit circle, one can always extend $m_1$ to a bigger cone
in $\mathbb{R}^2$. Furthermore, it is clear that we can extend $u(x)$
to a harmonic function
in the cone $\widetilde K $ generated by $\widetilde\Sigma$ using
(\ref
{u.from.m}).
We impose the following assumptions on the increments of the random walk:
\begin{itemize}
\item\textit{Normalization assumption}: We assume that $\mathbf EX_j=0,
\mathbf EX_j^2=1,j=1,\ldots,d$. In addition, we assume that
$\operatorname{cov}(X_i,X_j)=0$.
\item\textit{Moment assumption}: We assume that $\mathbf
E|X|^{\alpha}<\infty$ with $\alpha=p$ if $p>2$ and some $\alpha>2$
if $p\le2$.
\end{itemize}

\subsection{Tail distribution of \texorpdfstring{$\tau_x$}{$tau_x$} and a conditioned limit theorem}
Let
\[
K^\varepsilon=\bigl\{y\in\rd\dvtx \operatorname{ dist}(y,x)<\varepsilon|x| \mbox{ for
some } x \in K\bigr\}. %
\]
This new region is a cone. It follows from our assumptions
that we can pick a sufficiently small $\varepsilon>0$ which will
ensure that $K\subset K^{4\varepsilon}\subset\widetilde K$.
Recall that $u(x)$ is harmonic on a
bigger cone $\widetilde K$ and, therefore,
\[
u(x) \mbox{ is harmonic on } K^{4\varepsilon}. %
\]

Having $u$ we define a new function
\[
v(x)= \cases{ %
u(x),&\quad $x\in G,$
\vspace*{2pt}\cr
|x|^{p-a},&\quad $\mbox{otherwise},$}
\]
where
\[
G=K^\varepsilon\cap \bigl(K\cup\bigl\{x\in K^c\dvtx \operatorname{
dist}(x,\partial K)\leq |x|^{1-a}\bigr\} \bigr). %
\]
We will pick a sufficiently small constant $a>0$ later.

Let
%
\begin{equation}
\label{eq:defn.f} f(x)=\e v(x+X)-v(x),\qquad x\in K.
\end{equation}
Note that if $v(x+S(n))$ is a martingale, then $f(x)=0$.
Then let
%
\begin{equation}
\label{eq:defn.v} V(x)=v(x)-\e v\bigl(x+S(\tau_x)\bigr)+\e\sum
_{k=0}^{\tau_{x}-1}f\bigl(x+S(k)\bigr).
\end{equation}
It is not at all clear if function $V(x)$ is well defined. More
precisely, one has to show that $v(x+S(\tau_x))$ and
$\sum_{k=1}^{\tau_{x}-1}f(x+S(k))$ are integrable. Furthermore, one
has to show that $V$ does not depend on choice of $a$ and $\varepsilon$
from the definition of
$G$.

Finally, we define
\begin{eqnarray*}
K_+&:=& \bigl\{x\in K\dvtx \mbox{ there exists }\gamma>0\mbox{ such that for
every } R>0
\\
&&\hspace*{6pt}\mbox{there exists }n\mbox{ such that } \mathbf{P}\bigl(x+S(n)\in
D_{R,\gamma},\tau_x>n\bigr)>0 \bigr\},
\end{eqnarray*}
where
$D_{R,\gamma}:=\{x\in K\dvtx |x|\geq R, \operatorname{ dist}(x,\partial K)\geq\gamma
|x|\}$.

\begin{theorem}\label{T}
Assume the normalization as well as the moment assumption hold. Then,
for sufficiently small $a$,
the function $V$ is finite and harmonic for the
killed random walk $\{S(n)\}$, that is,
%
\begin{equation}
\label{Harm.F} \mathbf{E}\bigl[V\bigl(x+S(1)\bigr),\tau_x>1
\bigr]=V(x).
\end{equation}
The function $V(x)$ is strictly positive on the set $K_+$.
Moreover, as $n\to\infty$,
%
\begin{equation}
\label{eq.asym.tau} \mathbf P(\tau_x>n)\sim\varkappa V(x)
n^{-p/2},\qquad x\in K,
\end{equation}
where $\varkappa$ is an absolute constant.
\end{theorem}

%

Our moment assumption is optimal in the sense that the asymptotic behavior
of $\mathbf P(\tau_x>n)$ is in general different if $\mathbf
E|X|^p=\infty$. Indeed, consider a cone with $p>2$
and let $X$ be of the form $X=R\xi$, where $R$ is a nonnegative random
variable with
\[
\mathbf{P}(R>u)\sim u^{-\alpha}, \qquad\alpha\in(2,p) %
\]
and $\xi$ takes values on the unit sphere with some positive density on
$\Sigma$.
Clearly, $\mathbf E|X|^p=\infty$, that is, the moment assumption is not
fulfilled.
It follows from the structure of $X$ that
\[
\mathbf{P}(x+X\in D_{\sqrt{n},\gamma})\sim n^{-\alpha/2}\mathbf {P}(\xi \in
\Sigma\cap D_{0,\gamma})\geq cn^{-\alpha/2} %
\]
for some positive
$c$ and all sufficiently small $\gamma$.
Then
\[
\mathbf{P}(\tau_x>n)\geq\mathbf{P}\bigl(x+X(1)\in
D_{\sqrt{n},\gamma
}\bigr)\mathbf{P} \Bigl(\max_{k\leq n-1}\bigl|S(k)\bigr|<
\gamma\sqrt{n} \Bigr) \geq cn^{-\alpha/2}, %
\]
where in the last step we used the functional central limit theorem.
Therefore, the tail of $\tau_x$ is heavier than that of $\tau_x^{\mathrm{bm}}$.
We conjecture that this lower bound is precise, that is,
\[
\mathbf{P}(\tau_x>n)\sim\theta\mathbf{E}\tau_x
n^{-\alpha/2}. %
\]
Behind this relation is the well-known principle of one big jump: in
order to stay up to large moment of time $n$ inside the cone,
it is sufficient to make one big (of
order $\sqrt{n}$) jump into the inner part of the cone $K$ near time $0$.

We note that assumptions $\mathbf EX_j^2=1,j=1,\ldots,d$ and
$\operatorname{cov}(X_i,X_j)=0$ do not restrict the generality. More precisely, if
they are not
fulfilled and the covariance matrix of $X$ is positive-definite,
then there exists a matrix $M$ such that $Y=MX$ satisfies
these conditions. (If the covariance matrix is not positive-definite, then
the random walk lives on a hyperplane.)
Therefore, we have a random walk confined to a new
cone $M(K)=\{Mx, x\in K\}$. In the following example, we show the
influence of the correlation on the tail behavior of $\tau_x$.

\begin{example}
Consider a two-dimensional random walk with zero mean,
$\mathbf{E}X_1^2=\mathbf{E}X_2^2=1$ and
$\operatorname{cov}(X_1,X_2)=\rho\in(-1,1)$. Let $K$ be the positive quadrant,
that is, $K=\mathbb{R}_+^2$. In order to apply Theorem~\ref{T} we first
need to find a matrix $M$ such that the coordinates of the vector
$Y=MX$ become uncorrelated. Let $\varphi$ solve the equation $\sin
2\varphi=\rho$. Then the matrix
\[
M=\frac{1}{\sqrt{1-\rho^2}} \pmatrix{\cos\varphi&-\sin\varphi
\vspace*{2pt}\cr
-\sin\varphi&\cos\varphi}
\]
leads to uncorrelated coordinates. Therefore, $M(K)$ has opening
$\arccos(-\rho)$. Then, as it has been shown by Burkholder
\cite{Burk77}, $p=\pi/\arccos(-\rho)$. If\break
$\mathbf{E}|X|^{\pi/\arccos(-\rho)}$ is finite, then according to
Theorem~\ref{T}, we have
\[
\mathbf{P}(\tau_x>n)\sim V(x)n^{-\pi/2\arccos(-\rho)}\qquad\mbox{as }n\to
\infty. %
\]
It is worth mentioning that the minimal moment condition depends on
the covariance between $X_1$ and $X_2$.
\end{example}

If $V$ is harmonic for the random walk $MS(n)$ in the cone
$M(K)$, we have also a harmonic function for $S(n)$ in $K$. Indeed,
one can easily verify that $V(Mx)$ possesses this property.

We now turn to the asymptotic behavior of $S(n)$ conditioned to stay
in $K$. To state our results, we introduce the limit process.
For the
$d$-dimensional Brownian motion with starting point $x\in K$, one can
define a Brownian motion conditioned to stay in the cone via Doob's h-transform.
For that, we make a change of measure using the harmonic function $u$:
\[
\mathbf{\widehat P}_x^{(u)}\bigl(B(t)\in dy\bigr)=
\mathbf P\bigl(x+B(t)\in dy, \tau^{\mathrm{bm}}_x>t \bigr)
\frac{u(y)}{u(x)}. %
\]
This is possible since $u(x)>0$ inside the cone and $u(x+B(t\wedge\tau
^{\mathrm{bm}}_x))$ is a martingale.
Similarly, we define a random walk conditioned to stay in the cone $K$ by
\[
\mathbf{\widehat P}_x^{(V)}\bigl(S(n)\in dy\bigr)=
\mathbf P\bigl(x+S(n)\in dy, \tau_x>n \bigr)\frac{V(y)}{V(x)},\qquad x\in
\bigl\{x\dvtx V(x)>0\bigr\}. %
\]
This is possible due to Theorem~\ref{T}, where harmonicity of $V$ is proved.
We note also, that if we choose the starting point in $K_+$ then $S(n)$ under
$\mathbf{\widehat P}_x^{(V)}$ lives on $\{x\dvtx V(x)>0\}$, since this measure
does not allow transitions to the set $\{x\dvtx V(x)=0\}$.

%
\begin{theorem}\label{T2}
Assume that the normalization as well as the moment assumption are
fulfilled, then
%
\begin{equation}
\label{eq.cond.dist} \mathbf P \biggl(\frac{x+S(n)}{\sqrt n} \in\cdot \Big| \tau
_x>n \biggr)\to\mu\qquad\mbox{weakly},
\end{equation}
where $\mu$ is the probability measure on $K$ with the density
$H_0u(y)e^{-|y|^2/2}$, where $H_0$ is the normalizing constant.

Furthermore, for every $x\in K$, the process
$X^n(t)=\frac{S([nt])}{\sqrt n}$ under the probability measure
$\mathbf{\widehat P}^{(V)}_{x\sqrt n}$ converges weakly in the uniform
topology on $D[0,\infty)$
to the Brownian motion under the measure $\mathbf{\widehat P}^{(u)}_{x}$.
\end{theorem}

This is an extension of the classical theorems for one-dimensional
random walks conditioned
to stay positive by Iglehart~\cite{Ig74} and Bolthausen~\cite{Bolt76}.
Shimura \cite{Shim84} and Garbit \cite{Gar10} have proven similar
results for two-dimensional random walks in convex cones.

\begin{remark}\label{R1}
It is worth mentioning that one can prove Theorems \ref{T} and
\ref{T2} for more general cones if one restricts themselves to a smaller
class of random walks. If the jumps of $\{S(n)\}$ are bounded, then
the possibility to extend $u$ to a bigger cone is superfluous. In
this situation, one can show that $V$ is harmonic for arbitrary
starlike cone if we define $v$ by the relation
\[
v(x)=u(x_*+x)\qquad\mbox{for an appropriate }x_*\in K. %
\]
(For details, see Section~\ref{bounded}.) Having constructed the
harmonic function $V$, the proofs of all asymptotic statements from
Theorems \ref{T} and \ref{T2} do not require any change. As a result,
we get an asymptotic counterpart of the results proven by Varopoulos
\cite{Var99,Var00}. He derived upper and lower bounds for
probabilities $\mathbf{P}(\tau_x>n)$ and
$\mathbf{P}(x+S(n)=y,\tau_x>n)$ in terms of the harmonic function
$u$. All his bounds have the right order in $n$. In order to obtain
these estimates, he constructs superharmonic and subharmonic functions
for $\{S(n)\}$ in terms of $u$. (It is equivalent to construction of
super- and submartingale from $\{S(n)\}$.) And in order to obtain
asymptotic results we construct a harmonic function (martingale) for
the random walk. This is the main difference between our approach and
that of Varopoulos. 
\end{remark}

\subsection{Local limit theorems}
In this paragraph, we are going to determine the asymptotic behavior
of local probabilities for random walks conditioned to stay in a
cone. As it is usual in studying local probabilities, one has to
distinguish between lattice and nonlattice cases. We shall consider
lattice walks only, and analogous results for nonlattice walks can be
proved in the same way. The reason to choose the lattice case is an
application of the local limit theorems we prove here to lattice path
counting problems, which are very popular in combinatorics. Another
interesting application of local limit theorems could be the study of
the asymptotic behavior of the Green function for random walks in a
cone. This, combined with the Martin boundary theory, will allow to
find all harmonic functions.
\begin{itemize}
\item\textit{Lattice assumption}: $X$ takes values on a
lattice $R$ which is a nondegenerate linear transformation of $\mathbb{Z}^d$.
Furthermore, we assume that the distribution of $X$ is strongly aperiodic,
that is, for every $x\in R$, the smallest subgroup of $R$ which
contains the set
\[
\bigl\{y\dvtx y=x+z\mbox{ with some }z\mbox{ such that }\mathbf {P}(X=z)>0 \bigr
\} %
\]
is $R$ itself.
\end{itemize}

We first state a version of the Gnedenko local limit theorem.

%
\begin{theorem}\label{Loc.T1}
Assume that the assumptions of Theorem~\ref{T} and the lattice
assumption hold. Then
%
\begin{equation}
\label{Loc.T1.1} \sup_{y\in K}\biggl\llvert n^{p/2+d/2}
\mathbf{P} \bigl(x+S(n)=y,\tau _x>n \bigr)- \varkappa
V(x)H_0 u \biggl(\frac{y}{\sqrt{n}} \biggr)e^{-|y|^2/2n}\biggr
\rrvert \to0.\hspace*{-20pt}
\end{equation}
\end{theorem}

To prove a local limit theorem in the one-dimensional case, that is, for
random walks conditioned to stay positive, one starts usually from the
Wiener--Hopf factorization; see \cite{BD06,Car05,VW09}. Our approach
is completely different, and uses the integral limit theorem for
conditioned random walks (Theorem~\ref{T}) and a local limit theorem for
unconditioned random walks. Therefore, it works in all dimensions and
all cones, where Theorem~\ref{T} holds. In particular, our method
gives simple probabilistic proofs of local limit theorems for random
walks conditioned to stay positive.

We next find asymptotic behavior of
$\mathbf{P} (x+S(n)=y,\tau_x>n )$ for fixed $y$. [Note that
Theorem~\ref{Loc.T1} says only that this probability is
$o(n^{-p/2-d/2})$.]

\begin{theorem}\label{Loc.T2}
Under the assumptions of the preceding theorem, for every fixed
$y\in K$,
%
\begin{equation}
\label{Loc.T1.2} \mathbf{P} \bigl(x+S(n)=y,\tau_x>n \bigr)\sim
\varrho H_0^2\frac
{V(x)V'(y)}{n^{p+d/2}},
\end{equation}
where $V'$ is the harmonic function for the random walk $\{-S(n)\}$
and
\[
\varrho=\varkappa^2\int_K
u^2(w)e^{-|w|^2/2}\,dw. %
\]
Furthermore, for every $t\in(0,1)$ and any compact $D\subset K$
%
\begin{equation}
\label{Loc.T1.3} \mathbf{P} \biggl(\frac{x+S([tn])}{\sqrt{n}}\in D \Big|\tau
_x>n,x+S(n)=y \biggr)\to Q_t(D),
\end{equation}
where $Q_t$ is the measure on $K$ with the density
\[
\frac{1}{\rho(2\pi)^d}\frac{1}{ (t(1-t)
)^{p+d/2}}u^2(z)e^{-|z|^2/2t(1-t)}\,dz.
\]
\end{theorem}

In the next two subsections, we mention some interesting applications
of this theorem.
\subsection{Application to lattice paths enumeration}
Starting from the classical ballot problem, counting of lattice paths
confined to a certain domain, attracts a lot of attention. For lattice
paths in Weyl chambers associated with reflection groups, one often uses
a generalization of the classical reflection principle of Andre,
which has been proved by Gessel and Zeilberger \cite{GZ92}.
Unfortunately, the latter result can be applied only to some special
random walks
which are not allowed to jump over the boundary of the chamber.
Additionally, the
set of all possible steps should be invariant with respect to all
reflections. Grabiner and Magyar \cite{GM93} give the complete list of
all random walks to which the reflection principle can be
applied. Recently, the reflection principle of Gessel and Zeilberger
has been slightly generalized by Feierl \cite{F11+}. He derived a new
version of the reflection principle for random walks with steps which
are at most finite combinations of steps from the list of Grabiner and
Magyar. Another very popular cone is the positive quadrant in~$\mathbb{Z}^2$. Here, we mention papers of Bousquet-Melou \cite{BM05}
and of Bousquet-Melou and Mishna \cite{BM09}, where the authors
obtained exact results for some random walks with bounded steps in the
quarter plane. Raschel \cite{R09} and Kurkova and Raschel \cite{KR}
also considered a two-dimensional
random walk in the quarter plane, and proved some asymptotic results
for the exit position. All these papers are based on the analytical
approach suggested in the book of Fayolle, Iasnogorodski and Malyshev~\cite{FIM}. This method works for random walks on $\mathbb Z^2$
that can jump only to the nearest neighbors.

We next show how one can determine the asymptotic behavior of the
number of walks with endpoints $x$ and $y$ confined to a cone from our
results.

Consider lattice paths with the step set $\mathcal{S}=\{s_1,s_2,\ldots
,s_N\}$.
We assume that the corresponding random walk on $\mathbb{Z}^d$ is
strongly aperiodic.

If the vector sum of all $s_i$
is not equal to zero, one has to perform the Cramer transformation
with an appropriate parameter. For $R(h)=N^{-1}\sum_{i=1}^N
e^{(h,s_i)}$, we set
\[
\mathbf{P}(Y=s_i)=\frac{1}{NR(h)}e^{(h,s_i)}. %
\]
If there exists $h_0$ such that $\mathbf{E}Y=0$, then we have the following
formula for the number of walks with endpoints $x$ and $y$:
\[
N_n(x,y)=N^n\bigl(R(h_0)
\bigr)^ne^{(h_0,x-y)}\mathbf{P} \Biggl(x+\sum
_{k=1}^n Y(k)=y,\tau_x>n \Biggr).
\]
It is clear that $\mathbf{E}Y=0$ if and only if $R(h)$ attains its
minimum at $h=h_0$.
A necessary and sufficient condition for the existence of the global
minimum for $R$ is
that the step set is not contained in a closed half-space.

There exists a linear transformation with matrix $M$ such that $X=MY$
has uncorrelated coordinates and
\[
\mathbf{P}\Biggl(x+\sum_{k=1}^n Y(k)=y,
\tau_x>n\Biggr)=\mathbf{P}\bigl(Mx+S(n)=My,\tau_{x}>n
\bigr). %
\]
Since the number of possible steps is finite, we have a random walk
with bounded jumps. Therefore, we may use our results in any starlike
cone; see Remark~\ref{R1}. Applying Theorem~\ref{Loc.T2} to the random
walk $S(n)$ and cone $MK$, we obtain
%
\begin{equation}
\label{count} N_n(x,y)=C(x,y) \bigl(NR(h)\bigr)^n
n^{-p-d/2}\bigl(1+o(1)\bigr)\qquad \mbox{as }n\to\infty.
\end{equation}

It is worth mentioning that not only $C(x,y)$ but also $p$ may depend
on $h$. This means that $p$ depends not only on the cone $K$ but also
on the step set $\mathcal{S}$. An essential disadvantage of this
approach is the fact that we cannot give an explicit expression for
the function $C(x,y)$ and, therefore, we can only determine the rate of
growth of $N_n(x,y)$. Nevertheless, for large values of $x$ and $y$
inside the cone one can obtain an approximation for $C(x,y)$ from the
relation $V(x)\sim u(x)$.

We also note that upper bounds for $N_n(x,y)$ can be obtained from
the estimates due to Varopoulos. It follows from (0.7.4) in
\cite{Var99} that $C(x,y)$ from (\ref{count}) can be bounded from
above by $Cu(x+x_0)u(y+x_0)$ with some appropriate $x_0$. An essential
advantage of this bound consists in the fact that $u$ is more
accessible than~$V$.

Finally, we mention that our derivation of \eqref{count} is purely
probabilistic, since we use a strong approximation to prove
Theorem~\ref{T2}. And it is not at all clear how to prove
\eqref{count} by combinatorial methods. The only case known in the
literature are random walks with small steps in the quarter plane:
Fayolle and Raschel \cite{FR} derived a version of \eqref{count}
by means of the kernel method.
\subsection{Random walks in Weyl chambers}
As it has already been mentioned, random walks in Weyl chambers have
attracted a lot of attention in the recent past.

Let us first consider the chamber of type $A$, that is,
\[
W_A:=\bigl\{x\in\mathbb{R}^d\dvtx x_1<x_2<
\cdots<x_d\bigr\}. %
\]
In this case, one has $ u(x)=\prod_{i<j}(x_j-x_i)$ and
$p=d(d-1)/2$. $W_A$ is convex and $u$ is harmonic on the whole space
$\mathbb{R}^d$. Therefore, we may apply all our theorems to all random
walks satisfying normalization and moment conditions. If one additionally
assumes that the coordinates of $X$ are exchangeable, or even
independent, then $f(x)=0$. This has been shown by K\"onig, O'Connell
and Roch \cite{KOR02}. Therefore,
%
\begin{equation}
\label{EK_form} V(x)=u(x)-\mathbf{E}u\bigl(x+S(\tau_x)\bigr).
\end{equation}
This form of the harmonic function has been suggested by Eichelsbacher
and K\"onig \cite{EK08}. It is worth mentioning that if the
coordinates of $X$ are independent, then the moment condition from the
present paper is not optimal. It is shown in~\cite{DW10} that all the
statements in Theorems \ref{T} and \ref{T2} remain valid under the
condition $\mathbf{E}|X|^{d-1}<\infty$. For two further Weyl chambers,
\[
W_C:=\bigl\{x\in\mathbb{R}^d\dvtx 0<x_1<x_2<
\cdots<x_d\bigr\} %
\]
and
\[
W_D:=\bigl\{x\in\mathbb{R}^d\dvtx |x_1|<x_2<
\cdots<x_d\bigr\} %
\]
and random walks with independent coordinates K\"onig and Schmid
\cite{KS10} have proven versions of Theorems \ref{T} and \ref{T2}
under moment conditions which are weaker than
$\mathbf{E}|X|^{p}<\infty$. However, they have imposed an additional
symmetry condition. More precisely, they have assumed that some odd
moments of the distribution of coordinates are zero. This has been
done in order to make $u(x+S(n))$ a martingale and $f(x)=0$.
As a result, they had
the harmonic function of the form (\ref{EK_form}). One can verify that
all the statements of \cite{KS10} remain valid without the symmetry
condition mentioned above, if one takes the harmonic function from our
Theorem~\ref{T}. (One has first to show that this function is
well defined under the moment condition imposed by the authors of
\cite{KS10}.)

We next note that if our random walk has independent coordinates, then
Theorems \ref{Loc.T1} and \ref{Loc.T2} are valid for Weyl chamber
under weaker moment assumptions. Indeed, as we have already mentioned,
one needs an integral limit theorem for the conditioned random
walk. Therefore, the moment conditions from \cite{DW10} and
\cite{KS10} are sufficient for the validity of the local limit
theorems. Applying (\ref{Loc.T1.3}) to the Weyl chamber of type $A$,
we then see that the distribution of the excursion at time $tn$
converges to the measure determined by the density
\[
\frac{1}{\rho(2\pi)^d}\frac{1}{ (t(1-t) )^{p+d/2}} \biggl(\prod_{i<j}(z_j-z_i)
\biggr)^2e^{-|z|^2/2t(1-t)}\,dz,\qquad z\in W_A, %
\]
which is known to be the density of the distribution of eigenvalues in GUE.
This result corresponds to Theorem~1 of Baik and Suidan~\cite{BS07}.
\subsection{Description of our method} In the one-dimensional case, we
have only
two cones: positive and negative half-axis. To determine the behavior of
$\mathbf{P}(\tau_0>n)$ in this classical case one uses the Wiener--Hopf
factorization. For an arbitrary starting point $x$ inside one of the
half-axis, one
has $\mathbf{P}(\tau_x>n)\sim H(x)\mathbf{P}(\tau_0>n)$, where $H(x)$
is the renewal
function based on ladder heights. And one can easily infer that $H(x)$
is harmonic.
It is worth mentioning that the Wiener--Hopf method is quite powerful
as one does not
need to impose any moment conditions on the random walk.

Unfortunately, there are no general versions of the Wiener--Hopf
factorization for
the multidimensional case. We have found only two such attempts in the
literature.
First, Mogulskii and Pecherskii \cite{MP77} proved some factorization
identities for
random walks on semigroups. However, it is not clear how to get
asymptotics for exit times from
these identities. The second one is the paper by Greenwood and Shaked
\cite{GS77}, where a factorization
over a family of two-dimensional cones is performed. As a consequence,
the authors determined the
asymptotic behavior of some special first passage times.

The first step in the proof of Theorem~\ref{T} consists in the
construction of the harmonic
function $V$; see Section~\ref{sect.V.is.good}. Here, we use the
universality idea and
construct $V$ from the harmonic function $u$ for the Brownian motion:
\[
V(x)=\lim_{n\to\infty}\mathbf E \bigl[u\bigl(x+S(n)\bigr);
\tau_x>n \bigr],\qquad x\in K. %
\]
An additional difficulty arises from the fact that although
$u(x+B(t\wedge\tau^{\mathrm{bm}}_x ))$ is a martingale
the sequence $u(x+S(n\wedge\tau_x))$ is not. This explains correction
terms $f(x+S(k))$ in Lemma~\ref{lem-1}. Another difficulty is that in general we have $x+S(\tau
_x)\notin\partial K$
with positive probability. This is the reason for introducing the
extendability condition
on the cone $K$.

The second step of the proof is a coupling with the Brownian motion.
Although this idea
is quite natural, its naive application (starting from the beginning)
gives only rough asymptotics:
\[
\mathbf P(\tau_x>n)\sim n^{-p/2+o(1)},\qquad n\to\infty. %
\]
To obtain exact asymptotics, one has to wait until the random walk
moves far from
the boundary of the cone. In Lemma~\ref{lem1}, we show that this
happens with a high
probability. Then in Lemma~\ref{lem5} we couple out random walk with
the Brownian
motion using an extended version of Sakhanenko's coupling; see
Lemma~\ref{lem4}.
This allows us to obtain the exact asymptotics when starting at $y$ far
from the boundary,
\[
\mathbf P(\tau_y>n)\sim\varkappa u(y)n^{-p/2}. %
\]

Section~\ref{sect.asymptotics} is the final step of the proof of
Theorem~\ref{T}. We use the Markov property
at the first time $\nu_n$ when the random walk is far from the boundary
and the formula we obtained from the coupling in
Lemma~\ref{lem5}. Informally, this results in
\begin{eqnarray*}
\mathbf P(\tau_x>n) &\approx&\int\mathbf P(\tau_x>
\nu_n, S_{\nu
_n}\in dy)\mathbf P(\tau_y>n)
\\
&\approx&\varkappa\int\mathbf P(\tau_x>\nu_n,
S_{\nu_n}\in dy)u(y)n^{-p/2}
\\
&\approx&\varkappa V(x)n^{-p/2}.
\end{eqnarray*}
These relations are proved in Lemmas~\ref{lem11}, \ref{lem12}.
Proof of Theorem~\ref{T2} uses the same ideas.

It is worth mentioning that the method of constructing harmonic
functions for random walks
described above works also for Markov processes in discrete time. After
the first version of the
present paper was finished, we applied our approach to the following
two problems.

First, in \cite{DW12+} we found asymptotics for $\mathbf{P}(\tau
_{x,y}>n)$, where
$\tau_{x,y}=\min\{n\geq1\dvtx x+ny+\sum_{k=1}^n S(k)\leq0\}$ and
$S(k)$ is
a driftless random walk
with\break $\mathbf{E}|S(1)|^{2+\delta}<\infty$. The process $\sum_{k=1}^n
S(k)$ is not Markovian,
but one can obtain the Markov property increasing the dimension of the
process. More precisely,
$(\sum_{k=1}^n S(k),S(n))$ is a Markov chain, and consequently, $\tau
_{x,y}$ becomes the exit
time from the cone $\mathbb{R}_+\times\mathbb{R}$.

Second, in \cite{DKW}, our joint paper with Dima Korshunov, we
investigated the asymptotic
behavior of the stationary distribution for a positive recurrent Markov
chain on $\mathbb{R}_+$
with asymptotically zero drift. The crucial step was again a
construction of harmonic functions
for a chain killed at leaving an interval $[x_0,\infty)$, $x_0>0$.

Based on these two examples, we conjecture that our approach should
work for a wide class
of Markov chains, which converge, after an appropriate scaling, to
diffusion processes.
\section{Finiteness and positivity of $V$}\label{sect.V.is.good}
This section is devoted to the construction of the harmonic function
$V$. We consider first the case $d\geq2$. The one-dimensional case will
be considered in Section~\ref{one}.
\subsection{Finiteness}

We first derive some properties of the
functions $v(x)$ and~$f(x)$.

\begin{lemma}
\label{lem:bounds.u}
Let $u$ be harmonic on $K^{4\varepsilon}$ and $|u(x)|\le c|x|^p, x
\in K^{4\varepsilon}$. Then we have the following estimates for the
derivatives
%
\begin{eqnarray}
\label{eq:bound.u} \llvert u_{x_i}\rrvert &\le& C|x|^{p-1},\qquad x\in
K^{3\varepsilon},
\nonumber
\\
\llvert u_{x_ix_j}\rrvert &\le& C|x|^{p-2},\qquad x\in K^{2\varepsilon
},
\\
\llvert u_{x_ix_j x_k}\rrvert &\le & C|x|^{p-3},\qquad x\in
K^{\varepsilon}.\nonumber
\end{eqnarray}
\end{lemma}

Here, and throughout the text we denote as $C,c$ some generic
constants.

\begin{pf*}{Proof of Lemma \ref{lem:bounds.u}}
Since $u$ is harmonic on $K^{4\varepsilon}$ all its derivatives are
harmonic as well. Let $y\in K^{3\varepsilon}$. It immediately
follows from the definition of the cone $K^{4\varepsilon}$ that the
ball $B(y,\eta|y|)\subset K^{4\varepsilon}$ for $\eta=\varepsilon
/(1+3\varepsilon)$. Indeed, let $x$ be such that $\operatorname{ dist}(y,x)\le
3\varepsilon|x|$. Then,
since $|y|\le(1+3\varepsilon)|x|$, for $z\in B(y,\eta|y|)$,
\[
\operatorname{ dist}(z,x)\le\operatorname{ dist}(z,y)+\operatorname{ dist}(y,x) <\eta|y|+3\varepsilon|x|<4
\varepsilon|x|. %
\]
Hence, we can
apply the mean-value formula for harmonic functions to function
$u_{x_i}$ and obtain
\begin{eqnarray*}
|u_{x_i}|&=& \biggl\llvert \frac{1}{\operatorname{Vol}(B(y,\eta|y|))}\int_{B(y,\eta|y|)}
u_{x_i}\,dx \biggr\rrvert
\\
&=& \biggl\llvert \frac{1}{|\eta|y||^d\alpha(d)}\int_{\partial B(y,\eta|y|)} u
\nu_i \,ds \biggr\rrvert
\\
&\le&\frac{d\alpha(d)(\eta|y|)^{d-1}}{|\eta|y||^d\alpha(d)} \max_{x\in\partial B(y,\eta|y|)} u(x)
\\
&\le& c \frac{d}{\eta|y|}(1+\eta)^p |y|^p=c
\frac{d(1+\eta)^p}{\eta} |y|^{p-1}.
\end{eqnarray*}
Here, $\alpha(d)$ is the volume of the unit ball and we used the
Gauss--Green theorem. In the second line of the display, $\nu_i$ is
the outer
normal and integration takes place on the surface of the ball $B(y,\eta|y|)$.

The higher derivatives can be treated likewise. The claim of the lemma
immediately follows.
\end{pf*}

Next, we require a bound on $f(x)$.

\begin{lemma}
\label{lem:bound.f}
Let the assumptions of Lemma~\ref{lem:bounds.u} hold and $f$ be
defined by \eqref{eq:defn.f}. Let the moment and normalization assumptions
hold. Then, for some $\delta>0$,
\[
\bigl|f(x)\bigr| \le C |x|^{p-2-\delta}\qquad \mbox{for all } x\in K\mbox{ with }|x|\geq1.
\]
Furthermore,
\[
\bigl|f(x)\bigr|\leq C\qquad \mbox{for all } x\in K\mbox{ with }|x|\leq1. %
\]
\end{lemma}

\begin{pf}
Let $x\in K$ be such that $|x|\geq1$. Put $g(x)=|x|^{1-a}$, where we
pick constant $a$ later.
Fix some $\eta\in(0,\varepsilon)$ satisfying $\eta+\eta
^{1/(1-a)}\leq1$.
Then, for any $y\in B(0,\eta g(x))$, the sum $x+y\in G$.
By the Taylor theorem,
\[
\biggl\llvert u(x+y)-u(x)-\nabla u \cdot y -\frac{1}{2}\sum
_{i,j}u_{x_ix_j}y_iy_j
\biggr\rrvert \le R_3(x)|y|^3.
\]
The remainder can
be estimated by Lemma~\ref{lem:bounds.u}
\[
R_3(x)=\max_{z\in B(x,\eta g(x))}\max_{i,j,k}\bigl|u_{x_ix_jx_k}(z)\bigr|
\le C(1+\eta)^{p-3}|x|^{p-3}, %
\]
which will give us
%
\begin{equation}
\label{eq:taylor} \biggl\llvert u(x+y)-u(x)-\nabla u \cdot y -\frac{1}{2}
\sum_{i, j}u_{x_ix_j} y_iy_j
\biggr\rrvert \le C|x|^{p-3}|y|^3.
\end{equation}

Since $v=u$ on $G$, we can proceed as follows:
\begin{eqnarray*}
|f(x)| &=&\bigl|\e \bigl(u(x+X)-u(x) \bigr) \mathrm{1} \bigl(|X|
\le\eta g(x)\bigr) \bigr|
\\
&&{}+\bigl|\e \bigl(v(x+X)-v(x) \bigr) \mathrm{1} \bigl(|X|>\eta g(x)\bigr) \bigr|
\\
&\le&\biggl\llvert \e \biggl[ \biggl(\nabla u\cdot X +\frac{1}{2}\sum
_{i,j}u_{x_ix_j}X_iX_j
\biggr)\mathrm{1}\bigl(|X|\le\eta g(x)\bigr) \biggr]\biggr\rrvert
\\
&&{}+C|x|^{p-3}\e \bigl[|X|^3\mathrm{1}\bigl(|X|\le\eta g(x)\bigr)
\bigr]
\\
&&{} +C\e \bigl[|x|^p+\max\bigl(|X+x|^p,1\bigr)\mathrm{1}
\bigl(|X|>\eta g(x)\bigr) \bigr].
\end{eqnarray*}
Here we used also the bound $|v(z)|\leq C\max\{1,|z|^p\}$.

After rearranging the terms, we obtain
\begin{eqnarray*}
\bigl|f(x)\bigr| &\le& \biggl\llvert \e \biggl[\nabla u\cdot X +\frac{1}{2}\sum
_{i,j}u_{x_ix_j}X_iX_j
\biggr]\biggr\rrvert
\\
&&{}+\biggl\llvert \e \biggl[ \biggl(\nabla u\cdot X +\frac{1}{2}\sum
_{i,j}u_{x_ix_j}X_iX_j
\biggr)\mathrm{1}\bigl(|X|>\eta g(x)\bigr) \biggr]\biggr\rrvert
\\
&&{}+C|x|^{p-3}\e \bigl[|X|^3\mathrm{1}\bigl(|X|\le\eta g(x)\bigr)
\bigr] \\
&&{}+C\e \bigl[\bigl(|x|^p+\max\bigl(|X+x|^p,1\bigr)
\bigr)\mathrm{1}\bigl(|X|>\eta g(x)\bigr) \bigr].
\end{eqnarray*}
Now note that the first term is $0$ due to $\mathbf EX_i=0$, $\operatorname{cov}(X_i,
X_j)=0$ and $\Delta u=0$. The partial derivatives of the function $v$
in the second term are estimated via Lemma~\ref{lem:bounds.u}. As a
result,
\begin{eqnarray*}
\bigl|f(x)\bigr|&\le& C \bigl(|x|^{p-1}\mathbf E \bigl[|X|;|X|>\eta g(x) \bigr]
+|x|^{p-2}\mathbf E \bigl[|X|^2;|X|>\eta g(x) \bigr]
\\
&&\hspace*{13pt}{}+|x|^{p-3}\mathbf E \bigl[|X|^3;|X|\le\eta g(x) \bigr]
+|x|^{p}\mathbf P\bigl(|X|>\eta g(x)\bigr)
\\
&&\hspace*{123pt}{}+\mathbf E \bigl[\max\bigl(|X|^p,1\bigr);|X|> \eta g(x) \bigr]
\bigr).
\end{eqnarray*}
Hence, from the Markov inequality and
\[
\mathbf E \bigl[\max\bigl(|X|^p,1\bigr);|X|> \eta g(x) \bigr]\leq
\mathbf E \bigl[|X|^p;|X|> \eta g(x) \bigr]+ |x|^p
\mathbf{P}\bigl(|X|> \eta g(x)\bigr) %
\]
we conclude
%
\begin{eqnarray}
\label{new1}\qquad
\bigl|f(x)\bigr|&\le& C\frac{|x|^{p}}{\eta^2 g^2(x)}\mathbf E
\bigl[|X|^2;|X|>\eta g(x) \bigr] +C|x|^{p-3}\mathbf E
\bigl[|X|^3;|X|\le\eta g(x) \bigr]
\nonumber
\\[-8pt]
\\[-8pt]
\nonumber
&&{}+C\mathbf E \bigl[|X|^p;|X|> \eta g(x) \bigr].
\end{eqnarray}
Now recall the moment assumption that $\mathbf
E|X|^{2+2\delta}<\infty$ for some $\delta>0$. The first term is
estimated via the Chebyshev inequality,
\[
\frac{|x|^{p}}{\eta^2 g^2(x)}\mathbf E \bigl[|X|^2;|X|>\eta g(x) \bigr]\leq
\frac{|x|^{p}}{\eta^{2+2\delta} g^{2+2\delta}(x)}\mathbf E |X|^{2+2\delta}. %
\]
The second term can be estimated similarly,
\[
|x|^{p-3}\mathbf E \bigl[|X|^3;|X|\le\eta g(x) \bigr]\leq
|x|^{p-3}\eta^{1-2\delta}g^{1-2\delta}(x)\mathbf E
|X|^{2+2\delta}.
\]
Choosing $a$ sufficiently small, we see that the expectations in the
first line
of (\ref{new1}) are bounded by $C|x|^{p-2-\delta}$. In order to bound
the last term
in (\ref{new1}) we have to distinguish between $p\leq2$ and $p>2$.

If $p\leq2$, then, by the Chebyshev inequality,
\[
\mathbf{E} \bigl[|X|^p;|X|>\eta g(x) \bigr]\leq \frac{1}{(\eta g(x))^{2+2\delta-p}}
\mathbf{E} \bigl[|X|^{2+2\delta
} \bigr] \leq C|x|^{p-2-\delta} %
\]
for all $a$ sufficiently small.

In case $p>2$ we have, according to our moment condition, $\mathbf
{E}[|X|^p]<\infty$.
Consequently,
\[
\mathbf{E} \bigl[|X|^p;|X|>\eta g(x) \bigr]\leq C. %
\]
The second statement follows easily from the fact that $v(x)$ is
bounded on $|x|\leq1$ and
the inequality $\mathbf{E}[v(x+X)]\leq C(1+\mathbf{E}[|X|^p])$.
\end{pf}

\begin{lemma}
\label{v-lemma}
For any $x\notin K$,
\[
\bigl|v(x)\bigr|\leq C\bigl(1+|x|^{p-a}\bigr). %
\]
\end{lemma}

\begin{pf}
If $x\notin G$, then the inequality follows from the definition of
$v$. Assume now that $x\in G\setminus K$. If $|x|\leq1$, then
$|v(x)|$ is clearly bounded. But if $|x|>1$, then $\operatorname{
dist}(x,\partial K)\leq|x|^{1-a}$. And it follows from the Taylor
formula (recall that $v|_{\partial K}=0$) and Lemma~\ref{lem:bounds.u} that
%
\begin{equation}
\label{bound.ineq} \bigl|v(x)\bigr|\leq C|x|^{p-1}\operatorname{ dist}(x,\partial K)\leq
C|x|^{p-a}.
\end{equation}
Thus, the proof is finished.
\end{pf}

\begin{lemma}
\label{moments}
For every $\beta<p$, we have
%
\begin{equation}
\label{tau-mom} \mathbf{E}\bigl[\tau_x^{\beta/2}\bigr]\leq C
\bigl(1+|x|^\beta\bigr)
\end{equation}
and
%
\begin{equation}
\label{max-mom} \mathbf{E}\bigl[M^\beta(\tau_x)\bigr]\leq
C\bigl(1+|x|^\beta\bigr),
\end{equation}
where $M(\tau_x):=\max_{k\leq\tau_x}|x+S(k)|$.
\end{lemma}

This is the statement of Theorem~3.1 of \cite{MC84}. One has only to
notice that
$e(\Gamma,R)$ in that theorem is denoted by $p$ in our paper.

Next, we need to define an auxiliary process. Let
%
\begin{eqnarray}\label{eq:defn.y}
Y_0&=&v(x);
\nonumber
\\[-8pt]
\\[-8pt]
\nonumber
Y_{n+1}&=&v\bigl(x+S(n+1)\bigr)-\sum
_{k=0}^n f\bigl(x+S(k)\bigr),\qquad x\in K, n\ge0.
\end{eqnarray}

\begin{lemma}\label{lem-1}
The sequence $Y_n$ defined in \eqref{eq:defn.y} is a martingale.
\end{lemma}

\begin{pf}
The integrability of the sequence $Y_n$ is immediate from the bound
$u(x)\leq C|x|^p$ and
from Lemmas \ref{lem:bound.f} and \ref{v-lemma}.
Further,
\begin{eqnarray*}
\mathbf{E} [Y_{n+1}-Y_n|\mathcal{F}_n ] &=&
\mathbf{E} \bigl[ \bigl(v\bigl(x+S(n+1)\bigr)-v\bigl(x+S(n)\bigr)-f\bigl(x+S(n)
\bigr) \bigr)|\mathcal {F}_n \bigr]
\\
&=&-f\bigl(x+S(n)\bigr)+\mathbf{E} \bigl[ \bigl(v\bigl(x+S(n+1)\bigr)-v
\bigl(x+S(n)\bigr) \bigr)|S(n) \bigr]
\\
&=&-f\bigl(x+S(n)\bigr)+f\bigl(x+S(n)\bigr)=0,
\end{eqnarray*}
where we used the definition of the function $f$ in
\eqref{eq:defn.f}.
\end{pf}

\begin{lemma}\label{lem0}
For sufficiently small $a>0$, the function $V$ from (\ref{eq:defn.v})
is well defined. Furthermore,
%
\begin{equation}
\label{eq00} V(x)=\lim_{n\to\infty}\mathbf E \bigl[u\bigl(x+S(n)
\bigr);\tau_x>n \bigr],\qquad x\in K.
\end{equation}
This equality implies that $V$ does not depend on the choice of $a$ and
$\varepsilon$
in the definition of $G$.
\end{lemma}

\begin{pf}
First, using \eqref{eq:defn.y} we obtain,
\begin{eqnarray*}
\mathbf E \bigl[v\bigl(x+S(n)\bigr);\tau_x>n \bigr] & =& \mathbf
E[Y_n;\tau_x>n]+\sum_{l=0}^{n-1}
\mathbf E\bigl[f\bigl(x+S(l)\bigr);\tau _x>n\bigr]
\\
& =&\mathbf EY_n -\mathbf E[Y_n;\tau_x\le
n]+\sum_{l=0}^{n-1}\mathbf E\bigl[f
\bigl(x+S(l)\bigr);\tau_x>n\bigr].
\end{eqnarray*}
Since $Y_k$ is a martingale, $\mathbf EY_n=\mathbf EY_0=v(x)$ and
$\mathbf E[Y_n;\tau_x\le n]=\mathbf E[Y_{\tau_x};\tau_x\le n]$.
Using the definition of $Y_n$ once again, we arrive at
\begin{eqnarray*}
\mathbf E \bigl[v\bigl(x+S(n)\bigr);\tau_x>n\bigr] &=& v(x)-\mathbf
E \bigl[v\bigl(x+S(\tau_x)\bigr),\tau_x\leq n\bigr]
\\
&&{} +\mathbf E \Biggl[\sum_{l=0}^{\tau_x-1} f
\bigl(x+S(l)\bigr);\tau_x\le n \Biggr] \\
&&{}+\sum
_{l=0}^{n-1}\mathbf E\bigl[f\bigl(x+S(l)\bigr);
\tau_x>n\bigr].
\end{eqnarray*}
Combining Lemmas \ref{v-lemma} and \ref{moments}, we obtain
%
\begin{equation}
\label{nom1} \mathbf{E}\bigl|v\bigl(x+S(\tau_x)\bigr)\bigr|\leq
\mathbf{E}M^{p-a}({\tau_x})\leq C\bigl(1+|x|^{p-a}
\bigr).
\end{equation}
Then the dominated convergence theorem implies that
%
\begin{equation}
\label{dom.conv} \mathbf E \bigl[v\bigl(x+S(\tau_x)\bigr),
\tau_x\leq n\bigr]\to\mathbf{E}v\bigl(x+S(\tau_x)
\bigr).
\end{equation}
To estimate the third and fourth terms, it is sufficient to prove that
%
\begin{equation}
\label{eq:sum.finite} \mathbf E \Biggl[\sum_{l=0}^{\tau_x-1}
\bigl|f\bigl(x+S(l)\bigr)\bigr| \Biggr]\le C\bigl(1+|x|^{p-\delta}\bigr).
\end{equation}
Indeed, the dominated convergence theorem then implies that
\[
\mathbf E \Biggl[\sum_{l=0}^{\tau_x-1} f
\bigl(x+S(l)\bigr);\tau_x\le n \Biggr] \to \mathbf E \Biggl[\sum
_{l=0}^{\tau_x-1} f\bigl(x+S(l)\bigr) \Biggr]
\]
and
\[
\Biggl\llvert \sum_{l=0}^{n-1}\mathbf E
\bigl[f\bigl(x+S(l)\bigr);\tau_x>n\bigr]\Biggr\rrvert \le \mathbf E
\Biggl[\sum_{l=0}^{\tau_x-1} \bigl|f\bigl(x+S(l)
\bigr)\bigr|;\tau_x>n \Biggr]\to0 %
\]
since $\tau_x$ is finite a.s.

Hence, it remains to prove (\ref{eq:sum.finite}). Consider first the case
$p>2$. Assuming that $\delta<p-2$ and using Lemma~\ref{lem:bound.f},
we get
\[
\mathbf E \Biggl[\sum_{l=0}^{\tau_x-1} \bigl|f
\bigl(x+S(l)\bigr)\bigr| \Biggr]\leq C\mathbf{E}\bigl[\tau_x
M^{p-2-\delta}(\tau_x)\bigr]. %
\]
Applying H\"older's inequality with $p'<p/2$ and $q'<p/(p-2-\delta)$ and
Lemma~\ref{moments}, we obtain
\[
\mathbf E \Biggl[\sum_{l=0}^{\tau_x-1} \bigl|f
\bigl(x+S(l)\bigr)\bigr| \Biggr]\leq \bigl(\mathbf{E}\tau_x^{p'}
\bigr)^{1/p'} \bigl(\mathbf {E}M^{q'(p-2-\delta)}(\tau_x)
\bigr)^{1/q'}<C\bigl(1+|x|^{p-\delta}\bigr). %
\]
Such a choice of $p'$ and $q'$ is possible since
$(p/2)^{-1}+(p/(p-2-\delta))^{-1}<1$.
This proves (\ref{eq:sum.finite}) for $p>2$.

Next, consider the case $p\le2$. We split the sum in
(\ref{eq:sum.finite}) into four parts,
\begin{eqnarray*}
\mathbf E \Biggl[\sum_{l=0}^{\tau_x-1} \bigl|f
\bigl(x+S(l)\bigr)\bigr| \Biggr]& =& f(x)+\sum_{l=1}^{\infty}
\mathbf E \bigl[\bigl |f\bigl(x+S(l)\bigr)\bigr|;\tau _x>l \bigr]
\\
&=& f(x)+\sum_{l=1}^{\infty} \mathbf E \bigl[
\bigl|f\bigl(x+S(l)\bigr)\bigr|;\bigl|x+S(l)\bigr|\le 1,\tau_x>l \bigr]
\\
&&{}+ \sum_{l=1}^{\infty} \mathbf E \bigl[\bigl |f
\bigl(x+S(l)\bigr)\bigr|;1<\bigl|x+S(l)\bigr|\le\sqrt l, \tau_x>l \bigr]
\\
&&{}+ \sum_{l=1}^{\infty} \mathbf E \bigl[ \bigl|f
\bigl(x+S(l)\bigr)\bigr|;\bigl|x+S(l)\bigr|> \sqrt l, \tau_x>l \bigr]
\\
&=:&f(x)+\Sigma_1+\Sigma_2+\Sigma_3.
\end{eqnarray*}

According to Theorem~6.2 of \cite{Es68},
%
\begin{equation}
\label{esseen} \sup_{z\in\mathbb{R}^d}\mathbf{P}\bigl(\bigl|S(n)-z\bigr|\leq1\bigr)
\leq Cn^{-d/2}.
\end{equation}

By Lemma~\ref{lem:bound.f}, $|f(y)|\le C$ for $|y|\le1$. From this
bound and (\ref{esseen}), we obtain
\begin{eqnarray*}
\Sigma_1&\le &C\sum_{l=1}^\infty
\mathbf P\bigl(\bigl|x+S(l)\bigr|\le1,\tau_x>l\bigr)
\\
&\le& C\sum_{l=1}^\infty\mathbf P(
\tau_x>l/2)\sup_{y}\mathbf P\bigl(\bigl|y+S(l/2)\bigr|
\le1\bigr)
\\
&\le& C \sum_{l=1}^\infty l^{-d/2}
\mathbf P(\tau_x>l/2)
\\
&\le& C\mathbf{E}\bigl[\tau_x^{(p-\delta)/2}\bigr] \sum
_{l=1}^\infty l^{-d/2-(p-\delta)/2}\le C
\bigl(1+|x|^{p-\delta}\bigr),
\end{eqnarray*}
where the sum is convergent due to $d\ge2$.

Second, by Lemma~\ref{lem:bound.f},
\begin{eqnarray*}
\Sigma_2&\le &C\sum_{l=1}^\infty
\mathbf E \bigl[\bigl|x+S(l)\bigr|^{p-2-\delta
};1\le\bigl|x+S(l)\bigr|\le\sqrt l,
\tau_x>l \bigr]
\\
&\le& C \sum_{l=1}^\infty\sum
_{j=1}^{\sqrt l} \mathbf E \bigl[\bigl|x+S(l)\bigr|^{p-2-\delta};j
\le\bigl|x+S(l)\bigr|\le j+1,\tau_x>l \bigr]
\\
&\le& C \sum_{l=1}^\infty\sum
_{j=1}^{\sqrt l} j^{p-2-\delta}\mathbf P\bigl(j
\le\bigl|x+S(l)\bigl|\le j+1,\tau_x>l\bigr).
\end{eqnarray*}
Now we note that
\[
\mathbf P\bigl(j\le\bigl|x+S(l)\bigl|\le j+1,\tau_x>l\bigr) \le\mathbf P(
\tau_x>l/2)\sup_{y}\mathbf P\bigl(\bigl|y+S(l/2)\bigr|
\in[j,j+1]\bigr).
\]
Covering the region $\{z\dvtx |z|\in[j,j+1]\}$ by $Cj^{d-1}$ unit balls and
using (\ref{esseen}), we get
\[
\mathbf P\bigl(j\le\bigl|x+S(l)\bigr|\le j+1,\tau_x>l\bigr)\le C
j^{d-1}l^{-d/2}\mathbf P(\tau_x>l/2). %
\]
Then
\begin{eqnarray*}
\Sigma_2&\le& C \sum_{l=1}^\infty
\sum_{j=1}^{\sqrt l} j^{p-2-\delta
}j^{d-1}l^{-d/2}
\mathbf P(\tau_x>l/2)
\\
&\le& C \sum_{l=1}^\infty l^{p/2-1-\delta/2}
\mathbf P(\tau_x>l/2)
\\
&\le& C\mathbf{E}\bigl[\tau_x^{(p-\delta)/2}\bigr]\le C
\bigl(1+|x|^{p-\delta}\bigr),
\end{eqnarray*}
by Lemma~\ref{moments}.

Third, by Lemma~\ref{lem:bound.f} and the fact that $p\le2$,
\begin{eqnarray*}
\Sigma_3& \le& C\sum_{l=1}^\infty
\mathbf E \bigl[\bigl|x+S(l)\bigl|^{p-2-\delta
};\bigl|x+S(l)\bigr|>\sqrt l,\tau_x>l
\bigr]
\\
&\le& C\sum_{l=1}^\infty l^{(p-2-\delta)/2}
\mathbf P(\tau_x>l)
\\
&\le& C\mathbf{E}\bigl[\tau_x^{(p-\delta)/2}\bigr]\le C
\bigl(1+|x|^{p-\delta}\bigr).
\end{eqnarray*}
\upqed\end{pf}

\subsection{Positivity} In this paragraph, we show that $V$ is strictly
positive on $K_+$ and prove some further properties of this function.

\begin{lemma}\label{positiv}
The function $V$ possesses the following properties:
\begin{longlist}[(a)]
\item[(a)] For any $\gamma>0, R>0$, uniformly in $x\in D_{R,\gamma}$ we
have $V(tx)\sim u(tx)$ as $t\to\infty$.
\item[(b)] For all $x\in K$, we have $V(x)\leq C(1+|x|^p)$.
\item[(c)] The function $V$ is harmonic for the killed random walk,
that is,
\[
V(x)=\mathbf{E} \bigl[V\bigl(x+S(n_0)\bigr),\tau_x>n_0
\bigr],\qquad x\in K, n_0\geq1. %
\]
\item[(d)] The function $V$ is strictly positive on $K_+$.
\item[(e)] If $x\in K$, then $V(x)\leq V(x+x_0)$, for all $x_0$ such
that $x_0+K\subset K$.
\end{longlist}
\end{lemma}

\begin{pf}
To prove the part (a), it suffices to note that $t^{-p}u(tx)=u(x)$,
$\inf_{x\in D_{R,\gamma}}u(x)>0$, and use bounds (\ref{nom1}), (\ref
{eq:sum.finite}). These
inequalities together with $|u(x)|\leq C|x|^p$ give the part (b).

It suffices to prove (c) for $n_0=1$, since for bigger values of $n_0$
one can then use the Markov
property of $S(n)$. It is clear that
\begin{eqnarray*}
&&\mathbf{E}\bigl[u\bigl(x+S(n+1)\bigr),\tau_x>n+1\bigr]\\
&&\qquad=\int
_K\mathbf{P}\bigl(x+S(1)\in dy,\tau_x>1
\bigr) \mathbf{E}\bigl[u\bigl(y+S(n)\bigr),\tau_y>n\bigr].
\end{eqnarray*}
According to Lemma~\ref{lem0}, $\mathbf{E}[u(y+S(n)),\tau_y>n]\to V(y)$
for every $y\in K$. Furthermore,
it follows from (\ref{nom1}), (\ref{eq:sum.finite}) that $\mathbf
{E}[u(y+S(n)),\tau_y>n]\leq C(1+|y|^p)$.
This allows one to apply the dominated convergence theorem, which gives
\[
V(x)=\lim_{n\to\infty}\mathbf{E}\bigl[u\bigl(x+S(n+1)\bigr),
\tau_x>n+1\bigr]=\int_K\mathbf {P}
\bigl(x+S(1)\in dy,\tau_x>1\bigr)V(y). %
\]

To prove the positivity of $V(x)$, assume that $x\in K_+$.
Then for every $R>0$ there exists $n_0=n_0(R)$ such that
$\mathbf{P}(x+S(n_0)\in D_{R,\gamma},\tau_x>n_0)>0$ with
some $\gamma=\gamma(x)$. According to the first part of
the lemma, there exist $R>0$ such that
$\inf_{y\in D_{R,\gamma}}V(y)>0$. Consequently,
\begin{eqnarray*}
V(x)&=&\mathbf{E}\bigl[V\bigl(x+S(n_0)\bigr);\tau_x>n_0
\bigr]
\\
&\geq&\mathbf{E} \bigl[V\bigl(x+S(n_0)\bigr),x+S(n_0)
\in D_{R,\gamma},\tau _x>n_0 \bigr]>0.
\end{eqnarray*}

To prove (e), we first show that the same property holds for $u(x)$.
Indeed, if $x_0$ is such that $x_0+K\subset K$, then
\[
\bigl\{\tau_x^{\mathrm{bm}}>t \bigr\}\subset \bigl\{\tau
_{x+x_0}^{\mathrm{bm}}>t \bigr\} \qquad\mbox{for all }x\in K, t>0.
\]
Then, in view of (\ref{BS}),
\[
\varkappa u(x)=\lim_{t\to\infty}t^{p/2}\mathbf{P}\bigl(\tau
_x^{\mathrm{bm}}>t\bigr)\leq \lim_{t\to\infty}t^{p/2}
\mathbf{P}\bigl(\tau_{x+x_0}^{\mathrm{bm}}>t\bigr)=\varkappa
u(x+x_0). %
\]
Applying now Lemma~\ref{lem0}, we get
\begin{eqnarray*}
V(x)&=&\lim_{n\to\infty}\mathbf E \bigl[u\bigl(x+S(n)\bigr);
\tau_x>n \bigr]
\\
&\leq&\lim_{n\to\infty}\mathbf E \bigl[u\bigl(x+x_0+S(n)
\bigr);\tau _{x+x_0}>n \bigr]=V(x+x_0).
\end{eqnarray*}
Thus, the proof is finished.
\end{pf}

\subsection{An alternative construction of a harmonic function for
random walks with bounded jumps}\label{bounded}
In this paragraph, we show that $V$ remains well-defined and a strictly
positive harmonic function for random walks with bounded jumps if we
take $v(x)=u(x+x_*)$.

Assume that $\mathbf{P}(|X|\leq R)=1$ and let $x_*$ satisfy the
condition
\[
\operatorname{ dist}(x,\partial K)>R \qquad\mbox{for every }x\in K_*:=x_*+K. %
\]
(One can choose $x_*=t_*x_0$ with sufficiently large $t_*$.)
Therefore, $f(x)=\mathbf{E}v(x+X)-v(x)$ is well defined and the
statement of Lemma~\ref{lem:bound.f} is valid with $\delta=1$. This
implies, by the same arguments as in the proof of Lemma~\ref{lem0},
that
\[
\mathbf{E} \Biggl[\sum_{l=1}^{\tau_x-1}\bigl|f
\bigl(x+S(l)\bigr)\bigr| \Biggr]<C\bigl(1+|x|^{p-\delta}\bigr). %
\]
To show that $v(x+S(\tau_x))$ is integrable, we assume that
\[
u(x)\leq C|x|^{p-\delta}\operatorname{ dist}(x,\partial K). %
\]
(If $K$ is convex, then this inequality holds with $\delta=1$, see
\cite{Var99}, formula (0.2.3).) Since $\operatorname{
dist}(x_*+x+S(\tau_x),\partial K)$ is bounded, then in view of
Lemma~\ref{moments}
\[
\mathbf{E}v\bigl(x+S(\tau_x)\bigr)\leq C\mathbf{E}\bigl|x+S(
\tau_x)\bigr|^{p-\delta}<C\bigl(1+|x|^{p-\delta}\bigr).
\]
Thus, $V$ is well defined. Repeating the derivation of (\ref{eq00}),
we obtain
\[
V(x)=\lim_{n\to\infty}\mathbf{E}\bigl[u\bigl(x+x_0+S(n)
\bigr);\tau_x>n\bigr]. %
\]
This relation implies that $V$ is harmonic. The positivity follows
from Lem\-ma~\ref{positiv}.

Formally, $V$ might depend on $x_*$. But one can show, using the
coupling with the Brownian motion from the next section, that $V$ is
independent of $x_*$. It is sufficient to note that one can replace
$u(y)\sim u(x_*+y)$ under the conditions of Lemma~\ref{lem5} below.
\subsection{Construction of harmonic function in the one-dimensional
case}\label{one}
If $d=1$, then $K=(0,\infty)$. Random walks confined to the positive
half-line are well studied in the literature.
The main tool is the Wiener--Hopf factorization. This method allows one
to construct the harmonic function for any
oscillating random walk. It turns out that the ladder heights renewal
function is harmonic for $S(n)$ killed at
leaving $(0,\infty)$.

For the sake of completeness, we indicate how our method works for
one-dimensional random walks.

The harmonic function for the killed Brownian motion is $u(x)=x\mathrm{
1}_{\mathbb{R}_+}(x)$. We extend it to a harmonic
function on the whole axis by putting $u(x)=x$, $x\in\mathbb{R}$. Since
$u(x+S(n))$ is a martingale, the corrector
function $f\equiv0$. Therefore,
\[
V(x)=u(x)-\mathbf{E} \bigl[u\bigl(x+S(\tau_x)\bigr) \bigr].
\]
This function is strictly positive on $K$. It is well defined provided
that the expectation
$\mathbf{E}|u(x+S(\tau_x))|$ is finite. The latter property can be
shown by constructing an appropriate
positive supermartingale. Namely, put
\[
h(x)= \cases{ (R+x)^{1-a},&\quad $x> 0,$\vspace*{2pt}
\cr
|x|,&\quad $x\le0.$}
\]
Then, after some computations, one can show that for sufficiently large
$R$ and sufficiently small $a$ the
process $h(x+S(n\wedge\tau_x))$ is a positive supermartingale provided
$\mathbf E|X(1)|^{2+\delta}<\infty$ for
some $\delta>0$. Hence, by the optional stopping theorem
\[
\mathbf{E}\bigl|x+S(\tau_x)\bigr|\le(R+x)^{1-a}. %
\]
This shows the finiteness of $\mathbf{E}|u(x+S(\tau_x))|$. In addition,
this estimate implies that
$V(x)\sim x$ as $x\to\infty$.

\section{Coupling}

Let $\varepsilon>0$ be a constant and let
%
\begin{equation}
\label{defn.K} K_{n,\varepsilon}=\bigl\{x\in K\dvtx \operatorname{ dist}(x,\partial K)\ge
n^{1/2-\varepsilon}\bigr\}.
\end{equation}
Define
\[
\nu_n:=\min\bigl\{k\geq1\dvtx x+S(k)\in K_{n,\varepsilon}\bigr\}.
\]

\begin{lemma}
\label{lem1}
There exists a positive constant $C$ such that, for every $\varepsilon>0$,
\[
\mathbf{P}\bigl(\nu_n>n^{1-\varepsilon},\tau_x>n^{1-\varepsilon}
\bigr)\leq \exp\bigl\{ -Cn^{\varepsilon}\bigr\}. %
\]
\end{lemma}

\begin{pf}
Set, for brevity, $b_n=[n^{1/2-\varepsilon}]$, where $a$ is a positive number.

Clearly,
\begin{eqnarray*}
&&\mathbf P\bigl(\nu_n>n^{1-\varepsilon},\tau_x>n^{1-\varepsilon}
\bigr) \\
&&\qquad \le \mathbf P\bigl(x+S\bigl(b_n^2\bigr),x+S
\bigl(2b_n^2\bigr),\ldots,x+S\bigl(\bigl[n^\varepsilon
\bigr]b^2_n\bigr)\in K\setminus K_{n,\varepsilon}\bigr)
\\
&&\qquad\le \Bigl( \sup_{y\in K\setminus K_{n,\varepsilon}} \mathbf P\bigl(y+S
\bigl(b_n^2\bigr)\in K\setminus K_{n,\varepsilon}\bigr)
\Bigr)^{[n^\varepsilon]}.
\end{eqnarray*}
It follows from the scaling property of the cone that
\[
\sup_{y\in K\setminus K_{n,\varepsilon}} \mathbf P\bigl(y+S\bigl(b_n^2
\bigr)\in K\setminus K_{n,\varepsilon}\bigr) =\sup_{y\in K\setminus K_{1,\varepsilon}}
\mathbf P \biggl(y+\frac
{S(b_n^2)}{n^{1/2-\varepsilon}}\in K\setminus K_{1,\varepsilon} \biggr).
\]
Therefore, it is sufficient to show that the right-hand side is separated
from $1$. To this end, recall that there exists $x_0$ with $|x_0|=1$
such that $x_0+K\subset K$ and $\operatorname{ dist}(x_0+K,\partial K)>0$.
Then, by the scaling property of cones, for sufficiently large $t_0$,
the distance $\operatorname{ dist}(t_0x_0+K,\partial K)\ge1$.
Hence, $t_0x_0+K\subset K_{1,\varepsilon}$.

Let $d(x)=\operatorname{ dist}(x+K,K^c)$. Since the boundary of the cone is
continuous this function is continuous. We assumed that $d(x_0)>0$. Therefore,
the set $K_0=\{x\dvtx d(x)>0\}$ is open and nonempty. Since $K$ is a cone,
the set $K_0$ is a cone as well.
Since $K_0+K\subset K$, we have $y+K_0\subset K$ for every $y\in K$.
Consequently,
\[
t_0x_0+y+K_0\subset
t_0x_0+K\subset K_{1,\varepsilon} %
\]
for all $y\in K$.
This relation yields
\[
\sup_{y\in K\setminus K_{1,\varepsilon}} \mathbf P \biggl(y+\frac
{S(b_n^2)}{n^{1/2-\varepsilon}}\in K
\,\Big\backslash\, K_{1,\varepsilon} \biggr) \leq1-\mathbf{P} \biggl(\frac{S(b_n^2)}{n^{1/2-\varepsilon}}\in
t_0x_0+K_0 \biggr). %
\]
Further, by the central limit theorem,
\begin{eqnarray*}
\lim_{n\to\infty}\mathbf{P} \biggl(\frac
{S(b_n^2)}{n^{1/2-\varepsilon
}}\in
t_0x_0+K_0 \biggr) =\mathbf{P}
\bigl(B(1)\in t_0x_0+K_0 \bigr).
\end{eqnarray*}
Since $K_0$ is open, the probability $\mathbf{P} (B(1)\in
t_0x_0+K_0 )$
is strictly positive. This completes the proof of the lemma.
\end{pf}

\begin{remark}\label{conv=star}
We used in the proof of the last lemma that every convex cone is
starlike. Now we prove this fact.
Fix some $x_0\in\Sigma$. Then, due to convexity, $x_0+K\subset K$.
Assume that there exists $y\in K$
such that $\operatorname{ dist}(x_0+y,\partial K)< \operatorname{ dist}(x_0,\partial K)$.
Let $x\in\partial K$
satisfy $\operatorname{ dist}(x_0+y,\partial K)=\operatorname{ dist}(x_0+y,x)$. Using the
convexity of $K$ once again,
we see that there exists a hyperplane $H(x)$
such that $H(x)\cap K=\varnothing$ and $\operatorname{ dist}(x_0+y,H(x))<\operatorname{
dist}(x_0,\partial K)$.
But then $\operatorname{ dist}(x_0+y, H(x))<\operatorname{ dist}(x_0,H(x))$, and this
implies that the half-line
$\{x_0+ty,y>0\}$ cuts $H(x)$ and leaves the cone $K$, what contradicts
the fact that $x_0+K\subset K$.
\end{remark}

\begin{lemma}\label{lem2}
For every $\varepsilon>0$ the inequality
\[
\mathbf{E}\bigl[u\bigl(x+S\bigl(n^{1-\varepsilon}\bigr)\bigr);
\nu_n>n^{1-\varepsilon},\tau _x>n^{1-\varepsilon}\bigr]
\leq C(x)\exp\bigl\{-Cn^{c_{\varepsilon}}\bigr\} %
\]
holds.
\end{lemma}

\begin{pf}
Since $\nu_n>n^{1-\varepsilon}$ and $\tau_x>n^{1-\varepsilon}$,
\[
\operatorname{ dist}\bigl(x+S\bigl(n^{1-\varepsilon}\bigr),\partial K\bigr)\le
n^{1/2-\varepsilon}. %
\]
Therefore, applying the Taylor formula (and recalling that $u$ vanishes on
the boundary), we obtain
\begin{eqnarray*}
u\bigl(x+S\bigl(n^{1-\varepsilon}\bigr)\bigr)&\le& C\bigl|x+S\bigl(n^{1-\varepsilon}
\bigr)\bigr|^{p-1}\operatorname{ dist}\bigl(x+S\bigl(n^{1-\varepsilon}\bigr),\partial K
\bigr)
\\
&\le& C\bigl|x+S\bigl(n^{1-\varepsilon}\bigr)\bigr|^{p-1}n^{1/2-\varepsilon}.
\end{eqnarray*}
Hence, by the H\"older inequality,
\begin{eqnarray*}
&&\mathbf{E}\bigl[u\bigl(x+S\bigl(n^{1-\varepsilon}\bigr)\bigr);
\nu_n>n^{1-\varepsilon},\tau _x>n^{1-\varepsilon}\bigr]
\\
&&\qquad\le C n^{1/2-\varepsilon}\mathbf{E}\bigl[\bigl|x+S\bigl(n^{1
-\varepsilon
}
\bigr)\bigr|^{p-1};\nu _n>n^{1-\varepsilon},\tau_x>n^{1-\varepsilon}
\bigr]
\\
&&\qquad\le C n^{1/2}\mathbf{E}\bigl[\bigl|x+S\bigl(n^{1-\varepsilon}
\bigr)\bigr|^{p}\bigr]^{(p-1)/p}\mathbf P\bigl(\nu_n>n^{1-\varepsilon},
\tau_x>n^{1-\varepsilon}\bigr)^{1/p}.
\end{eqnarray*}
An application of Lemma~\ref{lem1} and a classical martingale bound
\[
\mathbf E\bigl|S\bigl(n^{1-\varepsilon}\bigr)\bigr|^p\le Cn^{p/2}
\]
gives the required exponential bound.
\end{pf}

We start by formulating an estimate of the quality of the normal
approximation of high-dimensional random walks which follows from
a result of G{\"e}ttse and Zaitsev \cite{GZ10}; see Theorem~4 there.

\begin{lemma}
\label{lem4}
If $\mathbf{E}|X|^{2+\delta}<\infty$ for some $\delta\in(0,1)$, then
one can define a random walk with the same distribution as $S(n)$
and a Brownian motion $B(t)$ on the same probability
space such that, for any $\gamma$ satisfying
$0<\gamma<\frac{\delta}{2(2+\delta)}$,
%
\begin{equation}
\label{L4} \mathbf{P} \Bigl(\sup_{u\leq n}\bigl|S\bigl([u]
\bigr)-B(u)\bigr|\geq n^{1/2-\gamma
} \Bigr)\leq Cn^{2\gamma+\gamma\delta-\delta/2}.
\end{equation}
\end{lemma}

\begin{pf}
According to Theorem~4 and (1.13) of \cite{GZ10}, one can construct on
a joint
probability space a copy of $S(n)$ and a standard Gaussian random walk $W(n)$
satisfying
\begin{eqnarray*}
&&\mathbf{P} \biggl(\max_{k\leq n}\bigl|S(k)-W(k)\bigr|\geq
\frac
{1}{2}n^{1/2-\gamma
} \biggr)
\\
&&\qquad\leq C \biggl(\frac{1}{2}n^{1/2-\gamma} \biggr)^{-(2+\delta)} n
\mathbf{E}\bigl|X(1)\bigr|^{2+\delta}
\\
&&\qquad\leq Cn^{2\gamma+\gamma\delta-\delta/2}.
\end{eqnarray*}
But, in view of the classical L\'evy construction of the Brownian
motion, we may
assume that there is a Brownian motion $B(t)$ on the same probability space
with the property $B(k)=W(k), k\geq0$. Therefore,
%
\begin{equation}
\label{lem4.1} \mathbf{P} \biggl(\max_{k\leq n}\bigl|S(k)-B(k)\bigr|\geq
\frac
{1}{2}n^{1/2-\gamma
} \biggr) \leq Cn^{2\gamma+\gamma\delta-\delta/2}.
\end{equation}

Moreover,
%
\begin{eqnarray}
\label{lem4.2}
\nonumber
\mathbf{P} \biggl(\sup_{u\leq n}\bigl|B(u)-B
\bigl([u]\bigr)\bigr|\geq\frac
{1}{2}n^{1/2-\gamma
} \biggr)&\leq& n\mathbf{P}
\biggl(\sup_{t\leq1}\bigl|B(t)\bigr|\geq\frac{1}{2}n^{1/2-\gamma
}
\biggr)
\\
&\leq& dn\mathbf{P} \biggl(\sup_{t\leq1}\bigl|B_1(t)\bigr|
\geq\frac{1}{2\sqrt {d}}n^{1/2-\gamma} \biggr)
\\
&\leq&\frac{4dn}{\sqrt{2\pi}} \int_{n^{1/2-\gamma}/2\sqrt
{d}}^\infty
e^{-u^2/2}\,du.\nonumber
\end{eqnarray}
In the last step, we used the reflection principle and the bound
\[
\mathbf{P} \Bigl(\sup_{t\leq1}\bigl|B_1(t)\bigr|\geq x
\Bigr)\leq2\mathbf {P} \Bigl(\sup_{t\leq1}B_1(t)\geq
x \Bigr). %
\]
By the triangle inequality,
\begin{eqnarray*}
\mathbf{P} \Bigl(\sup_{u\leq n}\bigl|S\bigl([u]\bigr)-B(u)\bigr|\geq
n^{1/2-\gamma
} \Bigr)&\le& \mathbf{P} \biggl(\max_{k\leq n}\bigl|S(k)-B(k)\bigr|
\geq\frac
{1}{2}n^{1/2-\gamma
} \biggr)
\\
&&{}+\mathbf{P} \biggl(\sup_{u\leq n}\bigl|B(u)-B\bigl([u]\bigr)\bigr|\geq
\frac
{1}{2}n^{1/2-\gamma} \biggr).
\end{eqnarray*}
Applying (\ref{lem4.1}) and (\ref{lem4.2}), we complete the proof.
\end{pf}

\begin{lemma}
\label{lem6}
There exists a finite constant $C$ such that
%
\begin{equation}
\label{L6.1} \mathbf{P}\bigl(\tau^{\mathrm{bm}}_{x}>t\bigr)\leq C
\frac{|x|^p}{t^{p/2}},\qquad x\in K.
\end{equation}
Moreover,
%
\begin{equation}
\label{L6.2} \mathbf{P}\bigl(\tau^{\mathrm{bm}}_{x}>t\bigr)\sim
\varkappa\frac{u(x)}{t^{p/2}},
\end{equation}
uniformly in $x\in K$ satisfying $|x|\le\theta_t\sqrt{t}$ with some
$\theta_t\to0$. Finally, the density $b_{t}(x,z)$ of the
probability $ \mathbf{P}(\tau^{\mathrm{bm}}_{x}>t, x+B(t)\in dz) $ is
%
\begin{equation}
\label{L6.3} b_t(x,z)\sim\varkappa_0
t^{-d/2 } e^{-|z|^2/(2t)}u(x)u(z) t^{-p}
\end{equation}
uniformly in $x,z \in K$ satisfying $|x|\le\theta_t\sqrt{t}$ and
$|z|\le\sqrt{t/\theta_t}$ with some \mbox{$\theta_t\to0$}.
\end{lemma}

These statements can be derived from estimates in \cite{BS97}.

\begin{pf*}{Proof of Lemma \ref{lem6}}
According to Theorem~1 of \cite{BS97},
%
\begin{equation}
\label{lem6.1} \mathbf{P}\bigl(\tau^{\mathrm{bm}}_x>t\bigr)=\sum
_{j=1}^\infty B_j \biggl(
\frac
{|x|^2}{2t} \biggr)^{a_j/2} { }_1F_1
\biggl(\frac{a_j}{2},a_j+\frac{d}{2},\frac{-|x|^2}{2t}
\biggr)m_j \biggl(\frac{x}{|x|} \biggr),
\end{equation}
where
\[
a_j:=\sqrt{\lambda_j+ \biggl(\frac{d}{2}-1
\biggr)^2}-\frac{d}{2}+1 %
\]
and
\[
B_j:=\frac{\Gamma ({(a_j+d)}/{2} )}{\Gamma
(a_j+
{d}/{2} )}\int_\Sigma
m_j(\theta)\,d\theta. %
\]
By the definition,
%
\begin{equation}
\label{lem6.2} { }_1F_1 (a,b,z )=1+
\frac{a}{b}\frac{z}{1!}+\frac
{a(a+1)}{b(b+1)}\frac{z^2}{2!}+\cdots.
\end{equation}
Then, for all $x\in K$ with $|x|^2\leq t$, we have
\[
{ }_1F_1 \biggl(\frac{a_j}{2},a_j+
\frac{d}{2},\frac{-|x|^2}{2t} \biggr)\leq e^{|x|^2/2t}\leq
e^{1/2}. %
\]
Furthermore, in view of Lemma~5 of \cite{BS97},
%
\begin{eqnarray}
\label{lem6.3} \bigl|m_j(\theta)\bigr|&\leq&\frac{C}{\sqrt{I_{a_j-1+d/2}(1)}}m_1(
\theta)
\nonumber
\\[-8pt]
\\[-8pt]
\nonumber
&\leq & C2^{a_j/2}\sqrt{\Gamma(a_j+d/2)}m_1(
\theta), \qquad\theta\in\Sigma,
\end{eqnarray}
where\vspace*{1pt} $I_\nu(x)=\sum_{m=0}^\infty\frac{1}{m!\Gamma(m+\nu
+1)}(x/2)^{\nu
+2m}$ is the modified Bessel function.
Applying (\ref{lem6.2}) and (\ref{lem6.3}) to the corresponding terms
in (\ref{lem6.1}), we obtain
\[
\mathbf{P}\bigl(\tau^{\mathrm{bm}}_x>t\bigr)\leq
Cm_1 \biggl(\frac{x}{|x|} \biggr) \sum
_{j=1}^\infty B_j 2^{a_j/2}\sqrt{
\Gamma(a_j+d/2)} \biggl(\frac
{|x|^2}{2t} \biggr)^{a_j/2}.
\]
Using the Stirling formula and (2.3) from \cite{BS97}, one can easily get
\[
B_j 2^{a_j/2}\sqrt{\Gamma(a_j+d/2)}\leq C
\lambda_j^{d/4}. %
\]
Consequently,
\[
\mathbf{P}\bigl(\tau^{\mathrm{bm}}_x>t\bigr)\leq
Cm_1 \biggl(\frac{x}{|x|} \biggr) \sum
_{j=1}^\infty\lambda_j^{d/4}
\biggl(\frac{|x|^2}{2t} \biggr)^{a_j/2}. %
\]
According to the Weyl asymptotic formula, see \cite{Chavel}, page~172,
\[
cj^{2/(d-1)}\leq\lambda_j\leq Cj^{2/(d-1)}. %
\]
This implies that
\[
\sum_{j=1}^\infty\lambda_j^{d/4}
\biggl(\frac{|x|^2}{2t} \biggr)^{a_j/2} \leq C \biggl(\frac{|x|^2}{2t}
\biggr)^{a_1/2} %
\]
for all $x$ satisfying $|x|^2\leq t$. Therefore,
%
\begin{equation}
\label{lem6.4} \mathbf{P}\bigl(\tau^{\mathrm{bm}}_x>t\bigr)\leq
Cm_1 \biggl(\frac{x}{|x|} \biggr) \biggl(\frac
{|x|^2}{2t}
\biggr)^{a_1/2} =C\frac{u(x)}{t^{p/2}}, \qquad |x|^2\leq t.
\end{equation}
This immediately implies that (\ref{L6.1}) holds.

The same arguments give also
\[
\sum_{j=2}^\infty B_j \biggl(
\frac{|x|^2}{2t} \biggr)^{a_j/2} { }_1F_1
\biggl(\frac{a_j}{2},a_j+\frac{d}{2},\frac{-|x|^2}{2t}
\biggr)m_j \biggl(\frac{x}{|x|} \biggr) \leq Cm_1
\biggl(\frac{x}{|x|} \biggr) \biggl(\frac{|x|^2}{2t} \biggr)^{a_2/2}.
\]
Since $a_2>a_1$,
\[
\mathbf{P}\bigl(\tau_x^{\mathrm{bm}}>t\bigr)\sim
B_1 \biggl(\frac{|x|^2}{2t} \biggr)^{a_1/2} {
}_1F_1 \biggl(\frac{a_1}{2},a_1+
\frac{d}{2},\frac{-|x|^2}{2t} \biggr)m_1 \biggl(
\frac{x}{|x|} \biggr) %
\]
uniformly in $|x|\leq\theta_t\sqrt{t}$. Noting that ${ }_1F_1
(\frac{a_1}{2},a_1+\frac{d}{2},\frac{-|x|^2}{2t} )\to1$
uniformly in $|x|\leq\theta_t\sqrt{t}$, we get (\ref{L6.2}).

According to Lemma~1 from \cite{BS97},
\[
b_t(x,z)=\frac{e^{-(|x|^2+|z|^{2})/2t}}{t|x|^{d/2-1}|z|^{d/2-1}}\sum_{j=1}^\infty
I_{a_j-1+d/2} \biggl(\frac{|x||z|}{t} \biggr)m_j \biggl(
\frac
{x}{|x|} \biggr) m_j \biggl(\frac{z}{|z|} \biggr).
\]
{F}rom the assumptions $|x|\leq\theta_t\sqrt{t}$ and $|z|\leq\sqrt {t/\theta_t}$, we get
uniform convergence as $\frac{|x||z|}{t}\to0$. Recalling the definition
of the Bessel functions and using (\ref{lem6.3}), we obtain
\[
b_t(x,z)\sim\frac{1}{\Gamma(a_1+d/2)}\frac
{e^{-(|x|^2+|z|^{2})/2t}}{t|x|^{d/2-1}|z|^{d/2-1}} \biggl(
\frac{|x||z|}{2t} \biggr)^{a_1-1+d/2}m_1 \biggl(
\frac
{x}{|x|} \biggr) m_1 \biggl(\frac{z}{|z|} \biggr)
\]
uniformly in $|x|\leq\theta_t\sqrt{t}$ and $|z|\leq\sqrt{t/\theta_t}$.
Simplifying this expression, and recalling the definitions of $p$ and
$u$, we get
\[
b_t(x,z)\sim\kappa_0 u(x)u(z)e^{-(|x|^2+|z|^{2})/2t}
t^{-p-d/2}. %
\]
Noting that $e^{-|x^2|/2t}\to1$, we obtain (\ref{L6.3}).
\end{pf*}

\begin{lemma}
\label{lem66}
If $K$ is convex then there exists a finite constant $C$ such that
\[
u(y)\geq C \bigl(\operatorname{ dist}(y,\partial K) \bigr)^p,\qquad y\in K.
\]
If $K$ is starlike and $C^2$, then
\[
u(y)\geq C|y|^{p-1}\operatorname{ dist}(y,\partial K),\qquad y\in K. %
\]
\end{lemma}

\begin{pf}
It is clear that
\[
\bigl\{\tau^{\mathrm{bm}}_y>t\bigr\}\supset\Bigl\{\sup
_{s\le t}\bigl|B(s)\bigr|<\operatorname{ dist}(y,\partial K)\Bigr\}. %
\]
Using the scaling property, we obtain
\[
\mathbf{P}\bigl(\tau^{\mathrm{bm}}_y>t\bigr)\geq\mathbf{P}
\biggl(\sup_{s\leq
1}\bigl|B(s)\bigr|<\frac
{\operatorname{
dist}(y,\partial K)}{\sqrt{t}} \biggr). %
\]
If $K$ is convex, then it has been proved in \cite{Var99}, see Theorem~1 and (0.4.1) there,
that
%
\begin{equation}
\label{L6.1a} \mathbf{P}\bigl(\tau^{\mathrm{bm}}_x>t\bigr)\leq C
\frac{u(x)}{t^{p/2}},\qquad x\in K,t>0.
\end{equation}

Using this bound with $t=(\operatorname{ dist}(y,\partial K))^2$,
we get
\[
\frac{u(y)}{(\operatorname{ dist}(y,\partial K))^p}\geq C\mathbf{P} \bigl(\tau^{\mathrm{bm}}_y>
\bigl(\operatorname{ dist}(y,\partial K)\bigr)^2 \bigr) \geq C \mathbf{P} \Bigl(
\sup_{s\leq1}\bigl|B(s)\bigr|<1 \Bigr). %
\]
Thus, the first statement is proved. The second one follows easily from
(0.2.1) in~\cite{Var99}.
\end{pf}

Using the coupling, we can translate the results of Lemma~\ref{lem6} to
the random walks setting when $y\in K_{n,\varepsilon}$.

\begin{lemma}\label{lem5}
For all sufficiently small $\varepsilon>0$,
%
\begin{equation}
\label{L6.4} \mathbf{P}(\tau_y>n)=\varkappa u(y)n^{-p/2}
\bigl(1+o(1)\bigr)\qquad \mbox{as }n\to \infty
\end{equation}
uniformly in $y\in K_{n,\varepsilon}$ such that $|y|\le\theta_n
\sqrt n$ for some $\theta_n\to0$. Moreover, there exists a constant
$C$ such that
%
\begin{equation}
\label{L6.5} \mathbf{P}(\tau_y>n)\le C \frac{|y|^p}{n^{p/2}},
\end{equation}
uniformly in $y\in K_{n,\varepsilon},n\ge1$. Finally, for any
compact set $D\subset K$,
%
\begin{equation}
\label{L6.6} \mathbf P\bigl(\tau_y>n, y+S(n)\in\sqrt n D\bigr)
\sim\varkappa_0 u(y) n^{-p/2} \int_D
\,dz\, e^{-|z|^2/2} u(z)
\end{equation}
uniformly in $y\in K_{n,\varepsilon}$ such that $|y|\le\theta_n\sqrt
n$ for some $\theta_n\to0$.
\end{lemma}

\begin{pf}
For every $y\in K_{n,\varepsilon}$ denote
\[
y^\pm=y\pm R_0x_0 n^{1/2-\gamma},
\]
where $x_0$ is such that $|x_0|=1$, $x_0+K\subset K$ and $R_0$ is such that
$\operatorname{ dist}(R_0x_0+K,\partial K)>1$. Note also that this choice of $R_0$
ensures that $R_0x_0n^{1/2-\gamma}\subset K_{n,\gamma}$.

If we take $\gamma>\varepsilon$, then for any $\varepsilon
'>\varepsilon$
there exists $n(\varepsilon')$ such that $y^{\pm}\in K_{n,\varepsilon
'}$ as soon
as $n\geq n(\varepsilon')$ and $y\in K_{n,\varepsilon}$.

Define
\[
A_n= \Bigl\{\sup_{u\leq n}\bigl|S\bigl([u]\bigr)-B(u)\bigr|\le
n^{1/2-\gamma} \Bigr\}, %
\]
where $B$ is the Brownian motion constructed in Lemma~\ref{lem4}.
The choice of $R_0$ ensures that $\tau^{\mathrm{bm}}_{y^+}>n$ on the set
$\{\tau_y>n\}\cap A_n$.
Then, using (\ref{L4}), we obtain
%
\begin{eqnarray}
\label{L5.1}
\mathbf{P}(\tau_y>n)&=&\mathbf{P}(
\tau_y>n,A_n)+o \bigl(n^{-r} \bigr)
\nonumber
\\[-8pt]
\\[-8pt]
\nonumber
&\leq&\mathbf{P}\bigl(\tau^{\mathrm{bm}}_{y^+}>n\bigr)+o
\bigl(n^{-r} \bigr),
\end{eqnarray}
where $r=r(\delta,\gamma)=\delta/2-2\gamma-\gamma\delta$. In the
same way, one can get
%
\begin{equation}
\label{L5.2} \mathbf{P}\bigl(\tau^{\mathrm{bm}}_{y^-}>n\bigr)\leq
\mathbf{P}(\tau_{y}>n)+o \bigl(n^{-r} \bigr).
\end{equation}
If $|y|\leq\theta_n\sqrt{n}$, then $|y^\pm|\leq\theta_n\sqrt {n}+R_0x_0n^{1/2-\gamma}=\theta_n'\sqrt{n}$.
Therefore, by Lem\-ma~\ref{lem6},
\[
\mathbf P\bigl(\tau^{\mathrm{bm}}_{y^\pm}>n\bigr)\sim\varkappa u
\bigl(y^\pm\bigr)n^{-p/2}. %
\]
It follows from the Taylor formula and Lemma~\ref{lem:bounds.u} that
%
\begin{equation}
\label{lbounda} \bigl|u\bigl(y^\pm\bigr)-u(y)\bigr|\leq C|y|^{p-1}\bigl|y^{\pm}-y\bigr|
\leq Cn^{p/2-\gamma}
\end{equation}
for all $y$ with $|y|\leq\sqrt{n}$. If $K$ is convex, then, according
to the
first part of Lemma~\ref{lem66},
%
\begin{equation}
\label{lbound} u(y)n^{-p/2}\ge C \bigl(\operatorname{ dist}(y,\partial K)
\bigr)^p n^{-p/2}\ge Cn^{-p\varepsilon},\qquad y\in
K_{n,\varepsilon}.
\end{equation}
If $K$ is not necessarily\vspace*{1pt} convex but $C^2$, then we may apply the
second part
of Lemma~\ref{lem66}, which gives the same estimate $u(y)\geq
Cn^{p(1/2-\varepsilon)}$.

Combining (\ref{lbounda}) and (\ref{lbound}), we obtain for $\gamma
>p\varepsilon$ an
estimate
\[
u\bigl(y^\pm\bigr)=u(y) \bigl(1+o(1)\bigr),\qquad y\in K_{n,\varepsilon},
|y|\leq\sqrt{n}. %
\]
Therefore, we have
\[
\mathbf{P}\bigl(\tau^{\mathrm{bm}}_{y^\pm}>n\bigr)=\varkappa
u(y)n^{-p/2}\bigl(1+o(1)\bigr). %
\]
{F}rom this relation and bounds (\ref{L5.1}) and (\ref{L5.2}), we
obtain
\[
\mathbf{P}(\tau_{y}>n)=\varkappa u(y)n^{-p/2}\bigl(1+o(1)
\bigr)+o \bigl(n^{-r} \bigr). %
\]
Thus, it remains to show that
%
\begin{equation}
\label{L5.3} n^{-r}=o\bigl(u(y)n^{-p/2}\bigr)
\end{equation}
for all sufficiently small $\varepsilon>0$ and all $y\in
K_{n,\varepsilon}$ with $|y|\leq\sqrt{n}$.

Using (\ref{lbound}), we see that (\ref{L5.3}) will be valid for all
$\varepsilon$ satisfying
\[
r=\delta/2-2\gamma-2\gamma\delta>p\varepsilon. %
\]
This proves (\ref{L6.4}). To prove (\ref{L6.5}), it is sufficient to
substitute (\ref{L6.1}) in (\ref{L5.1}).

The proof of (\ref{L6.6}) is similar. Define two sets,
\begin{eqnarray*}
D^+&=&\bigl\{z\in K\dvtx \operatorname{dist}(z, D)\le\bigl(|x_0|+1\bigr)n^{-\gamma}
\bigr\},
\\
D^-&=&\bigl\{z\in D\dvtx \operatorname{dist}(z, \partial D)\ge\bigl(|x_0|+1\bigr)n^{-\gamma}
\bigr\}.
\end{eqnarray*}
Clearly, $D^-\subset D\subset D^+$. Then, arguing as above, we get
%
\begin{eqnarray}
\label{L6.7}
&&\mathbf{P}\bigl(\tau_y>n,y+S(n)\in\sqrt n D
\bigr)\nonumber\\
&&\qquad\le \mathbf{P}\bigl(\tau_y>n,y+S(n)\in\sqrt n D,
A_n\bigr)+o \bigl(n^{-r} \bigr)
\nonumber
\\[-8pt]
\\[-8pt]
\nonumber
&&\qquad\le \mathbf{P}\bigl(\tau^{\mathrm{bm}}_{y^+}>n,y^++B(n)\in
\sqrt n D^+, A_n\bigr)+o \bigl(n^{-r} \bigr)
\\
&&\qquad\le\mathbf{P}\bigl(\tau^{\mathrm{bm}}_{y^+}>n,y^++B(n)\in\sqrt n D^+
\bigr)+o \bigl(n^{-r} \bigr).\nonumber
\end{eqnarray}
Similarly,
%
\begin{eqnarray}
\label{L6.8}&& \mathbf{P}\bigl(\tau_y>n,y+S(n)\in\sqrt n D\bigr)
\nonumber
\\[-8pt]
\\[-8pt]
\nonumber
&&\qquad\ge
\mathbf{P}\bigl(\tau^{\mathrm{bm}}_{y^-}>n,y^-+B(n)\in\sqrt n D^-
\bigr)+o \bigl(n^{-r} \bigr).
\end{eqnarray}
Now we apply (\ref{L6.3}) and obtain
\begin{eqnarray*}
\mathbf{P}\bigl(\tau^{\mathrm{bm}}_{y^\pm}>n,y^\pm+B(n)\in
\sqrt n D^\pm\bigr)&\sim& \varkappa_0 u
\bigl(y^\pm\bigr) \int_{\sqrt n D^\pm} \,dz
\,e^{-|z|^2/(2n)}u(z) n^{-{d}/{2}}n^{-p}
\\
&=&\varkappa_0 u\bigl(y^\pm\bigr) \int
_{ D^\pm} \,dz\, e^{-|z|^2/2}u(z) n^{-p/2}.
\end{eqnarray*}
It is sufficient to note now that
\[
u\bigl(y^\pm\bigr)\sim u(y)\quad \mbox{and}\quad \int_{ D^\pm}
\,dz \,e^{-|z|^2/2}u(z)\to \int_{ D} \,dz\, e^{-|z|^2/2}u(z)
\]
as $n\to\infty$. From these relations and bounds (\ref{L6.7}) and
(\ref{L6.8}), we obtain
\[
\mathbf{P}\bigl(\tau_y>n,y+S(n)\in\sqrt n D\bigr)= \bigl(
\varkappa_0+o(1)\bigr) u(y) \int_{ D} \,dz\,
e^{-|z|^2/2}u(z) n^{-p/2} +o \bigl(n^{-r} \bigr).
\]
Recalling (\ref{L5.3}), we arrive at the conclusion.
\end{pf}

%
\section{Asymptotics for \texorpdfstring{$\mathbf{P}(\tau_x>n)$}{$\mathbf{P}(tau_x>n)$}}\label{sect.asymptotics}
We first note that, in view of Lemma~\ref{lem1},
%
\begin{eqnarray}
\label{T1.1}
\mathbf{P}(\tau_x>n)&=&\mathbf{P}\bigl(
\tau_x>n,\nu_n\leq n^{1-\varepsilon
}\bigr)+\mathbf{P}
\bigl(\tau_x>n,\nu_n> n^{1-\varepsilon}\bigr)
\nonumber
\\[-8pt]
\\[-8pt]
\nonumber
&=&\mathbf{P}\bigl(\tau_x>n,\nu_n\leq n^{1-\varepsilon}
\bigr)+O \bigl(e^{-Cn^\varepsilon} \bigr).
\end{eqnarray}
Using the strong Markov property, we get the
following estimates for the first term:
%
\begin{eqnarray}
\label{T1.2}
&&\int_{K_{n,\varepsilon}}\mathbf{P} \bigl(x+S(
\nu_n)\in dy,\tau_x>\nu_n,\nu_n
\leq n^{1-\varepsilon} \bigr)\mathbf{P}(\tau_y>n) \nonumber\\
&&\qquad\leq\mathbf{P}
\bigl(\tau_x>n,\nu_n\leq n^{1-\varepsilon}\bigr)
\\
&&\qquad\leq\int_{K_{n,\varepsilon}}\mathbf{P} \bigl(x+S(\nu _n)\in
dy,\tau_x>\nu_n,\nu_n\leq n^{1-\varepsilon}
\bigr)\mathbf{P}\bigl(\tau_y>n-n^{1-\varepsilon}\bigr).\nonumber
\end{eqnarray}
Applying now Lemma~\ref{lem5}, we obtain
%
\begin{eqnarray}
\label{T1.3}
\nonumber
&&\mathbf{P}\bigl(\tau_x>n;
\nu_n\leq n^{1-\varepsilon}\bigr)
\\
\nonumber
&&\qquad= \frac{\varkappa+o(1)}{n^{p/2}}\mathbf{E} \bigl[u\bigl(x+S(\nu_n)
\bigr);\tau _x>\nu _n,\bigl|x+S(\nu_n)\bigr|\leq
\theta_n\sqrt{n},\nu_n\leq n^{1-\varepsilon
} \bigr]
\\
&&\qquad\quad{}+O \biggl(\frac{1}{n^{p/2}}\mathbf{E} \bigl[\bigl|x+S(\nu
_n)\bigr|^p;\tau_x>\nu_n,\bigl|x+S(
\nu_n)\bigr|> \theta_n\sqrt{n},\nu_n\leq
n^{1-\varepsilon} \bigr] \biggr)
\\
\nonumber
&&\qquad= \frac{\varkappa+o(1)}{n^{p/2}}\mathbf{E} \bigl[u\bigl(x+S(\nu_n)
\bigr);\tau _x>\nu _n,\nu_n\leq
n^{1-\varepsilon} \bigr]
\\
&&\qquad\quad{}+O \biggl(\frac{1}{n^{p/2}}\mathbf{E} \bigl[\bigl|x+S(\nu _n)
\bigr|^p;\tau_x>\nu_n,\bigl|x+S(
\nu_n)\bigr|> \theta_n\sqrt{n},\nu_n\leq
n^{1-\varepsilon} \bigr] \biggr).\nonumber
\end{eqnarray}
We now show that the first expectation converges to $V(x)$ and that
the second expectation is negligibly small.

\begin{lemma}\label{lem11}
Under the assumptions of Theorem~\ref{T},
\[
\lim_{n\to\infty}\mathbf{E} \bigl[u\bigl(x+S(\nu_n)
\bigr);\tau_x>\nu_n,\nu _n\leq
n^{1-\varepsilon} \bigr]=V(x). %
\]
\end{lemma}

\begin{pf}
By the definition of $Y_n$,
\[
u\bigl(x+S(\nu_n)\bigr)=Y_{\nu_n}+\sum
_{k=0}^{\nu_n-1}f\bigl(x+S(k)\bigr). %
\]
Consequently,
\begin{eqnarray*}
&&\mathbf{E} \bigl[u\bigl(x+S(\nu_n)\bigr);\tau_x>
\nu_n,\nu_n\leq n^{1-\varepsilon} \bigr]
\\
&&\qquad=\mathbf{E}
\bigl[Y_{\nu_n};\tau_x>\nu_n,\nu_n
\leq n^{1-\varepsilon
} \bigr]
\\
&&\qquad\quad{}+\mathbf{E} \Biggl[\sum_{k=0}^{\nu_n-1}f
\bigl(x+S(k)\bigr);\tau _x>\nu_n,\nu_n\leq
n^{1-\varepsilon} \Biggr].
\end{eqnarray*}
Recall that it was shown in Lemma~\ref{lem0} that
%
\begin{equation}
\label{domconv} \mathbf E\sum_{k=0}^{\tau_x-1}\bigl|f
\bigl(x+S(k)\bigr)\bigr|<\infty.
\end{equation}
Then, since $\nu_n\to\infty$,
\begin{eqnarray*}
\Biggl\llvert \mathbf{E} \Biggl[\sum_{k=0}^{\nu_n-1}f
\bigl(x+S(k)\bigr);\tau_x>\nu _n,\nu _n\leq
n^{1-\varepsilon} \Biggr]\Biggr\rrvert &\le& \mathbf{E} \Biggl[\sum
_{k=0}^{\tau_x-1}\bigl|f\bigl(x+S(k)\bigr)\bigr|;
\tau_x>\nu _n \Biggr]\\
&\to& 0. %
\end{eqnarray*}

Rearranging the terms, we have
%
\begin{eqnarray}
\label{T1.4}
\nonumber
\mathbf{E} \bigl[Y_{\nu_n};\tau_x>
\nu_n,\nu_n\leq n^{1-\varepsilon} \bigr] &=&\mathbf{E}
\bigl[Y_{\nu_n\wedge
n^{1-\varepsilon}};\tau_x>\nu_n\wedge
n^{1-\varepsilon},\nu_n\leq n^{1-\varepsilon} \bigr]
\\
&=&
\mathbf{E} \bigl[Y_{\nu_n\wedge
n^{1-\varepsilon}};\tau_x>
\nu_n\wedge n^{1-\varepsilon} \bigr]
\\
&&{}- \mathbf{E} \bigl[Y_{n^{1-\varepsilon}};\tau_x> n^{1-\varepsilon},
\nu_n>n^{1-\varepsilon} \bigr].\nonumber
\end{eqnarray}
Recalling the definition of $Y_n$, we get
\begin{eqnarray*}
&&\mathbf{E} \bigl[Y_{n^{1-\varepsilon}};\tau_x>n^{1-\varepsilon},\nu
_n>n^{1-\varepsilon} \bigr]\\
&&\qquad=\mathbf{E} \bigl[u\bigl(x+S
\bigl(n^{1-\varepsilon}\bigr)\bigr);\tau_x>n^{1-\varepsilon
},\nu
_n> n^{1-\varepsilon} \bigr]
\\
&&\qquad\quad{}-\mathbf{E} \Biggl[\sum_{k=0}^{n^{1-\varepsilon}-1}f
\bigl(x+S(k)\bigr);\tau _x>n^{1-\varepsilon},\nu_n>n^{1-\varepsilon}
\Biggr].
\end{eqnarray*}
The first term goes to zero due to Lemma~\ref{lem2}, the second term
vanishes by (\ref{domconv}) and
by the dominated convergence theorem. Therefore,
%
\begin{equation}
\label{T1.5} \mathbf{E} \bigl[Y_{n^{1-\varepsilon}};\tau_x>n^{1-\varepsilon},
\nu _n>n^{1-\varepsilon} \bigr]\to0.
\end{equation}
Further,
\begin{eqnarray*}
\mathbf{E} \bigl[Y_{\nu_n\wedge n^{1-\varepsilon}};\tau_x>\nu _n\wedge
n^{1-\varepsilon} \bigr] & =&\mathbf{E} [Y_{\nu_n\wedge
n^{1-\varepsilon}} ] -\mathbf{E}
\bigl[Y_{\nu_n\wedge n^{1-\varepsilon}};\tau_x\le\nu _n\wedge
n^{1-\varepsilon} \bigr]
\\
& =&\mathbf EY_0 -\mathbf{E} \bigl[Y_{\nu_n\wedge n^{1-\varepsilon}};
\tau_x\le\nu _n\wedge n^{1-\varepsilon} \bigr]
\\
& =&u(x) -\mathbf{E} \bigl[Y_{\tau_x};\tau_x\le
\nu_n\wedge n^{1-\varepsilon
} \bigr],
\end{eqnarray*}
where we have used the martingale property of $Y_{n}$. Noting that
$\nu_n\wedge n^{1-\varepsilon}\to\infty$ almost surely, we have
\[
Y_{\tau_x}\mathbf{1}\bigl\{\tau_x\leq\nu_n
\wedge n^{1-\varepsilon}\bigr\} \to Y_{\tau_x}. %
\]
Then, using the integrability of $Y_{\tau_x}$ [see \eqref{nom1} and
\eqref{eq:sum.finite},
and the dominated
convergence theorem], we obtain
%
\begin{equation}
\label{T1.6} \mathbf{E} \bigl[Y_{\tau_x};\tau_x\le
\nu_n\wedge n^{1-\varepsilon
} \bigr] \to \mathbf{E}Y_{\tau_x}
\end{equation}
Combining (\ref{T1.4})--(\ref{T1.6}), we obtain
\[
\mathbf{E} \bigl[u\bigl(x+S(\nu_n)\bigr);\tau_x>
\nu_n,\nu_n\leq n^{1-\varepsilon} \bigr]\to u(x)-
\mathbf{E}Y_{\tau_x}=V(x). %
\]
This proves the lemma.
\end{pf}

In what follows, we will use the Fuk--Nagaev inequalities several times.
For the reader's convenience, we state them in the following lemma.

\begin{lemma}\label{FN-ineq}
Let $\xi_i$ be independent identically distributed random variables
with $\mathbf{E}[\xi_1]=0$ and
$\mathbf{E}[\xi_1^2]<\infty$. Then, for all $x,y>0$,
%
\begin{equation}
\label{NF1} \mathbf{P} \Biggl(\sum_{i=1}^n
\xi_i\geq x,\max_{i\leq n}\xi_i\leq y
\Biggr)\leq e^{x/y} \biggl(\frac{n\mathbf{E}[\xi^2]}{xy} \biggr)^{x/y}
\end{equation}
and
%
\begin{equation}
\label{NF2} \mathbf{P} \Biggl(\sum_{i=1}^n
\xi_i\geq x \Biggr)\leq e^{x/y} \biggl(\frac{n\mathbf{E}[\xi^2]}{xy}
\biggr)^{x/y}+n\mathbf {P}(\xi>y).
\end{equation}
\end{lemma}

The second inequality is (1.56) from Corollary~1.11 of \cite{Nag79}.
The first one is not directly stated there, but it can be found in the proof
of Theorem~4 of \cite{FN}. There are no proofs in \cite{Nag79} and we
refer the
interested reader to the original paper \cite{FN}.

\begin{corollary}\label{FN-cor}
For all $x,y>0$,
%
\begin{equation}
\label{NF3} \mathbf{P} \Bigl(\bigl|S(n)\bigr|>x,\max_{k\leq n}\bigl|X(k)\bigr|\leq
y \Bigr)\leq2d e^{x/\sqrt{d}y} \biggl(\frac{\sqrt{d}n}{xy} \biggr)^{x/\sqrt{d}y}
\end{equation}
and
%
\begin{equation}
\label{NF4} \mathbf{P}\bigl(\bigl|S(n)\bigr|>x\bigr)\leq2d e^{x/\sqrt{d}y}
 \biggl(
\frac{\sqrt {d}n}{xy} \biggr)^{x/\sqrt{d}y}+n\mathbf{P}\bigl(\bigl|X(1)\bigr|>y
\bigr).
\end{equation}
\end{corollary}

\begin{pf}
It is clear that
\begin{eqnarray*}
\mathbf{P} \Bigl(\bigl|S(n)\bigr|>x,\max_{k\leq n}\bigl|X(k)\bigr|\leq y \Bigr)
&\leq&
\sum_{j=1}^d \mathbf{P}
\biggl(\bigl|S_j(n)\bigr|>\frac{x}{\sqrt d},\max_{k\leq
n}\bigl|X_j(k)\bigr|
\leq y \biggr)
\\
&\leq&\sum_{j=1}^d \mathbf{P}
\biggl(S_j(n)>\frac{x}{\sqrt d},\max_{k\leq
n}X_j(k)
\leq y \biggr)\\
&&{}+ \sum_{j=1}^d \mathbf{P}
\biggl(S_j(n)<-\frac{x}{\sqrt d},\min_{k\leq
n}X_j(k)
\geq-y \biggr).
\end{eqnarray*}
Applying now (\ref{NF1}) to every summand and recalling that $\mathbf
{E}[(X_j(1))^2]=1$,
we get~(\ref{NF3}). The bound (\ref{NF4}) follows from (\ref{NF3}) and
inequality
\[
\mathbf{P} \Bigl(\bigl|S(n)\bigr|>x,\max_{k\leq n}\bigl|X(k)\bigr|> y \Bigr)\leq
\mathbf{P} \Bigl(\max_{k\leq n}\bigl|X(k)\bigr|> y \Bigr)\leq n\mathbf{P}
\bigl(\bigl|X(1)\bigr|>y\bigr). %
\]
\upqed\end{pf}

\begin{lemma}\label{lem12}
Under the assumptions of Theorem~\ref{T},
\[
\lim_{n\to\infty}\mathbf{E} \bigl[\bigl|x+S(\nu_n)\bigr|^p;
\tau_x>\nu _n,\bigl|S(\nu_n)\bigr|>
\theta_n\sqrt{n},\nu_n\leq n^{1-\varepsilon} \bigr] =0.
\]
\end{lemma}

\begin{pf}
We take $\theta_n=n^{-\varepsilon/8}$. Let
\[
\mu_n:=\min\bigl\{j\ge1\dvtx \bigl|X(j)\bigr|>n^{1/2-\varepsilon/4}\bigr\}.
\]
Since $|S(\nu_n)|\le n^{3/2}$ on the event $\{\mu_n>\nu_n,\nu_n\leq
n^{1-\varepsilon}\}$, we arrive
at the following bound:
\begin{eqnarray*}
&& \mathbf{E} \bigl[\bigl|x+S(\nu_n)\bigr|^p;\tau_x>
\nu_n,\bigl|S(\nu_n)\bigr|> \theta _n\sqrt {n},
\nu_n\leq n^{1-\varepsilon},\mu_n>\nu_n \bigr]
\\
&&\qquad\le Cn^{p(3/2)}\mathbf P\bigl(\bigl|S(\nu_n)\bigr|>
\theta_n\sqrt {n},\nu _n\leq n^{1-\varepsilon},
\mu_n>\nu_n\bigr)
\\
&&\qquad\le Cn^{p(3/2)}\sum_{j=1}^{n^{1-\varepsilon}}
\mathbf P\bigl(\bigl|S(j)\bigr|> \theta_n\sqrt{n},\mu_n>j\bigr).
\end{eqnarray*}
Applying now (\ref{NF3}) with $x=\theta_n\sqrt{n}=n^{1/2-\varepsilon
/8}$, $y=n^{1/2-\varepsilon/4}$ to every probability term, we get
\begin{eqnarray*}
\sum_{j=1}^{n^{1-\varepsilon}}\mathbf P\bigl(\bigl|S(j)\bigr|>
\theta_n\sqrt {n},\mu_n>j\bigr) \leq2d \sum
_{j=1}^{n^{1-\varepsilon}} \biggl(\frac
{(ed)j}{n^{1-3\varepsilon/8}}
\biggr)^{n^{\varepsilon/8}/\sqrt{d}} \leq\exp\bigl\{-Cn^{\varepsilon/8}\bigr\}.
\end{eqnarray*}

As a result,
%
\begin{equation}
\label{T1.8} \mathbf{E} \bigl[\bigl|x+S(\nu_n)\bigr|^p;
\tau_x>\nu_n,\bigl|S(\nu_n)\bigr|> \theta
_n\sqrt {n},\nu_n\leq n^{1-\varepsilon},\mu_n>
\nu_n \bigr]\to0.
\end{equation}
Next,
\begin{eqnarray*}
&&\mathbf{E} \bigl[\bigl|x+S(\nu_n)\bigr|^p;\tau_x>
\nu_n,\bigl|S(\nu_n)\bigr|> \theta _n\sqrt {n},
\nu_n\leq n^{1-\varepsilon},\mu_n\le\nu_n
\bigr]
\\
&&\qquad\le \mathbf{E} \bigl[\bigl|x+S(\nu_n)\bigr|^p;
\tau_x>\mu_n,\nu_n\leq n^{1-\varepsilon
},
\mu_n\le\nu_n \bigr]
\\
&&\qquad\le \sum_{j=1}^{n^{1-\varepsilon}}\mathbf{E}
\bigl[\bigl|x+S(\nu_n)\bigr|^p;\tau _x>j,
\nu_n\leq n^{1-\varepsilon},j\le\nu_n,\mu_n=j
\bigr].
\end{eqnarray*}
For each $j$, we split the sum $S(\nu_n)$ in 3 parts
\[
\bigl|x+S(\nu_n)\bigr|^p\le C\bigl(\bigl|x+S(j-1)\bigr|^p+\bigl|X(j)\bigr|^p+\bigl|S(
\nu_n)-S(j)\bigr|^p\bigr). %
\]
Then
\begin{eqnarray*}
&& \sum_{j=1}^{n^{1-\varepsilon}}\mathbf{E} \bigl[\bigl|S(\nu
_n)-S(j)\bigr|^p;\tau _x>j,\nu_n\leq
n^{1-\varepsilon},j\le\nu_n,\mu_n=j \bigr]
\\
&&\qquad\le \sum_{j=1}^{n^{1-\varepsilon}}\mathbf{E}\bigl|M
\bigl(n^{1-\varepsilon
}\bigr)\bigr|^p\mathbf P(\tau_x>j-1,
\mu_n=j)
\\
&&\qquad\le C\sum_{j=1}^{n^{1-\varepsilon}}n^{(1-\varepsilon)p/2}
\mathbf P(\tau_x>j-1)\mathbf P\bigl(\bigl|X(j)\bigr|>n^{1/2-\varepsilon/4}\bigr)
\\
&&\qquad= C n^{(1-\varepsilon)p/2}\mathbf P\bigl(\bigl|X(1)\bigr|>n^{1/2-\varepsilon/4}\bigr) \sum
_{j=1}^{n^{1-\varepsilon}}\mathbf P(\tau_x>j-1).
\end{eqnarray*}
The bound $\mathbf{E}|M(n^{1-\varepsilon})|^p\le C n^{(1-\varepsilon)p/2}$
holds due to the Doob and Rosenthal inequalities for $p\ge2$ and additionally
H\"older's inequality for $p<2$.

There are two cases now. For $p>2$, the sum
\[
\sum_{j=1}^\infty\mathbf P(
\tau_x>j)<\infty, %
\]
since $\mathbf E\tau_x<\infty$. In addition, by the Chebyshev
inequality,
\[
n^{(1-\varepsilon)p/2}\mathbf P\bigl(|X|>n^{1/2-\varepsilon/4}\bigr) \le
n^{(1-\varepsilon)p/2}\frac{\mathbf
E|X|^p}{n^{p(1/2-\varepsilon/4)}}\le n^{-p\varepsilon/4}\mathbf
E|X|^p\to0. %
\]
Next, for $p\le2$, we use the fact that $\mathbf
E|X|^{2+\delta}<\infty$ for some $\delta>0$,
\[
n^{(1-\varepsilon)p/2}\mathbf P\bigl(|X|>n^{1/2-\varepsilon/4}\bigr) \le
n^{(1-\varepsilon)p/2}\frac{\mathbf
E|X|^{2+\delta}}{n^{(2+\delta)(1/2-\varepsilon/4)}}. %
\]
Since $\mathbf E\tau_x^{p/2-\beta}<\infty$, for any $\beta\in(0,p/2)$,
\[
\sum_{j=1}^{n^{1-\varepsilon}}\mathbf P(
\tau_x>j) \le\mathbf E\tau_x^{p/2-\beta}\sum
_{j=1}^n\frac{1}{j^{p/2-\beta}} \le
Cn^{1-p/2+\beta}. %
\]
Then
%
\begin{eqnarray}
\label{already_shown}
&&
\sum_{j=1}^{n^{1-\varepsilon}}n^{(1-\varepsilon)p/2}
\mathbf P(\tau_x>j-1)\mathbf P\bigl(|X|>n^{1/2-\varepsilon/4}\bigr) \nonumber\\
&&\qquad\le
C n^{1-p/2+\beta}n^{(1-\varepsilon)p/2}n^{-(2+\delta
)(1/2-\varepsilon/4)}
\\
&&\qquad=Cn^{\beta-\varepsilon(p-1)/2-\delta(1/2-\varepsilon/4)}\to0,\nonumber
\end{eqnarray}
once we pick sufficiently small $\varepsilon>0$ and $\beta>0$.
Therefore, in each case,
%
\begin{equation}
\label{eq:sum1} \sum_{j=1}^{n^{1-\varepsilon}}\mathbf{E}
\bigl[\bigl|S(\nu _n)-S(j)\bigr|^p;\tau _x>j,
\nu_n\leq n^{1-\varepsilon},j\le\nu_n,\mu_n=j
\bigr]\to0.
\end{equation}

Next, we analyze
\begin{eqnarray*}
&& \sum_{j=1}^{n^{1-\varepsilon}}\mathbf{E}
\bigl[\bigl|X(j)\bigr|^p;\tau _x>j,\nu _n\leq
n^{1-\varepsilon},j\le\nu_n,\mu_n=j \bigr]
\\
&&\qquad \le \sum_{j=1}^{n^{1-\varepsilon}}\mathbf{E}
\bigl[|X|^p;|X|>n^{1/2-\varepsilon/4} \bigr]\mathbf P(
\tau_x>j-1).
\end{eqnarray*}
As above, there are two cases: $p>2$ and $p\leq2$. If $p>2$, then we apply
\[
\sum_{j=1}^{\infty}\mathbf P(
\tau_x>j)<\infty, \qquad\mathbf{E} \bigl[|X|^p;|X|>n^{1/2-\varepsilon/4}
\bigr]\to0. %
\]
If $p\le2$, then
\begin{eqnarray*}
&&\sum_{j=1}^{n^{1-\varepsilon}}\mathbf{E}
\bigl[|X|^p;|X|>n^{1/2-\varepsilon/4} \bigr]\mathbf P(
\tau_x>j)
\\
&&\qquad\le Cn^{-p/2+\beta+1}\mathbf E|X|^{2+\delta} n^{-(2-p+\delta)(1/2-\varepsilon/4)}
\\
&&\qquad\le C n^{\beta-\delta(1/2-\varepsilon/4)+(2-p)\varepsilon/4}\to0
\end{eqnarray*}
once we pick sufficiently small $\varepsilon>0$ and $\beta>0$.
Therefore,
%
\begin{equation}
\label{eq:sum2} \sum_{j=1}^{n^{1-\varepsilon}}\mathbf{E}
\bigl[\bigl|X(j)\bigr|^p;\tau _x>j,\nu _n\leq
n^{1-\varepsilon},j\le\nu_n,\mu_n=j \bigr]\to0.
\end{equation}
Further,
%
\begin{eqnarray}
\label{3terms}
\nonumber
&&\sum_{j=1}^{n^{1-\varepsilon}}
\mathbf{E} \bigl[\bigl|x+S(j-1)\bigr|^p;\tau _x>j,
\nu_n\leq n^{1-\varepsilon},j\le\nu_n,\mu_n=j
\bigr]
\\
\nonumber
&&\qquad\le2^p|x|^p\mathbf{P}\bigl(\mu_n
\leq n^{1-\varepsilon}\bigr)\\
&&\qquad\quad{}+ 2^p\sum_{j=1}^{n^{1-\varepsilon}}
\mathbf{E} \bigl[\bigl|S(j-1)\bigr|^p;\tau _x>j,
\mu_n=j \bigr]\nonumber
\\
&&\qquad\le2^p|x|^p\mathbf{P}\bigl(\mu_n
\leq n^{1-\varepsilon}\bigr) \\
&&\qquad\quad{}+2^p\sum_{j=1}^{n^{1-\varepsilon}}
\mathbf{E} \bigl[\bigl|S(j-1)\bigr|^p;\bigl|S(j-1)\bigr|>n^{1/2-\varepsilon/8},
\mu_n=j \bigr]\nonumber
\\
&&\qquad\quad{}+ 2^p\sum_{j=1}^{n^{1-\varepsilon}}
\mathbf{E} \bigl[\bigl|S(j-1)\bigr|^p;\bigl|S(j-1)\bigr|\le n^{1/2-\varepsilon/8},
\tau_x>j,\mu_n=j \bigr].\nonumber
\end{eqnarray}
Using the Chebyshev inequality, we obtain
\[
\mathbf{P}\bigl(\mu_n\leq n^{1-\varepsilon}\bigr)\leq
n^{1-\varepsilon
}\mathbf {P}\bigl(|X|>n^{1/2-\varepsilon/4}\bigr) \leq
dn^{-\varepsilon/2}.
\]
For the second term in (\ref{3terms}) we note that on $\mu_n=j$ the sum
$|S(j-1)|\le n^{3/2}$. Hence,
\begin{eqnarray*}
&&\sum_{j=1}^{n^{1-\varepsilon}}\mathbf{E}
\bigl[\bigl|S(j-1)\bigr|^p;\bigl|S(j-1)\bigr|>n^{1/2-\varepsilon/8},\mu_n=j \bigr]
\\
&&\qquad\le n^{3p/2}\sum_{j=1}^{n^{1-\varepsilon}}
\mathbf {P}\bigl(\bigl|S(j-1)\bigr|>n^{1/2-\varepsilon/8},\mu_n=j\bigr)
\\
&&\qquad\le Cn^{3p/2+1} \exp\bigl\{-Cn^{\varepsilon/8}\bigr\},
\end{eqnarray*}
by the Fuk--Nagaev inequality (\ref{NF3}). The third term,
\begin{eqnarray*}
&&\sum_{j=1}^{n^{1-\varepsilon}}\mathbf{E}
\bigl[\bigl|S(j-1)\bigr|^p;\bigl|S(j-1)\bigr|\le n^{1/2-\varepsilon/8},\tau_x>j,
\mu_n=j \bigr]
\\
&&\qquad\le \sum_{j=1}^{n^{1-\varepsilon}}n^{p(1/2-\varepsilon/8)}
\mathbf {P}(\tau _x>j,\mu_n=j)
\\
&&\qquad\le n^{p(1/2-\varepsilon/8)}\mathbf{P}\bigl(|X|>n^{1/2-\varepsilon/4}\bigr) \sum
_{j=1}^{n^{1-\varepsilon}}\mathbf{P}(\tau_x>j-1)\to 0
\end{eqnarray*}
as has already been shown in \eqref{already_shown}. Hence,
%
\begin{equation}
\label{eq:sum3} \sum_{j=1}^{n^{1-\varepsilon}}\mathbf{E}
\bigl[\bigl|x+S(j-1)\bigr|^p;\tau _x>j,\nu _n\leq
n^{1-\varepsilon},j\le\nu_n,\mu_n=j \bigr]\to0.
\end{equation}
Now the claim follows from equations (\ref{T1.8}), (\ref{eq:sum1}),
(\ref{eq:sum2}) and (\ref{eq:sum3}).
\end{pf}

Now we are in position to complete the proof of Theorem~\ref{T}. It
follows from the lemmas and (\ref{T1.1}) and (\ref{T1.3}) that
\[
\mathbf{P}(\tau_x>n)=\frac{\varkappa V(x)}{n^{p/2}}\bigl(1+o(1)\bigr).
\]
%
\section{Weak convergence results}\label{sect.weak.convergence}

\begin{lemma}
For any $x\in K$, the distribution $\mathbf
P (\frac{x+S(n)}{\sqrt n} \in\cdot| \tau_x>n )$ weakly
converges to the distribution with the density
$H_0e^{-|y|^2/2}u(y)$, where $H_0$ is the normalizing
constant.
\end{lemma}

\begin{pf}
It suffices to show that, for any compact $A\subset K$,
%
\begin{eqnarray}
\label{eq.mmm} \frac{\mathbf P(x+S(n)\in\sqrt n A,\tau_x>n)}{\mathbf P(\tau
_x>n)}\to H_0\int_A
e^{-|y|^2/2}u(y)\,dy.
\end{eqnarray}
Take $\theta_n$ which goes to zero slower than any power function.
First note that, as in~(\ref{T1.1}) and (\ref{T1.3}),
\begin{eqnarray*}
&&\mathbf P\bigl(x+S(n)\in\sqrt n A,\tau_x>n\bigr)
\\
&&\qquad=\mathbf{P}\bigl(\tau_x>n,x+S(n)\in\sqrt n A,\nu_n
\leq n^{1-\varepsilon}\bigr)+O \bigl(e^{-Cn^\varepsilon} \bigr)
\\
&&\qquad=\mathbf{P}\bigl(\tau_x>n,x+S(n)\in\sqrt n A, \bigl|S(
\nu_n)\bigr|\le\theta_n\sqrt n,\nu_n\leq
n^{1-\varepsilon}\bigr)+o\bigl(\mathbf P(\tau_x>n)\bigr).
\end{eqnarray*}
In the last line, we used the following estimates which hold by the Markov
property, Lemmas~\ref{lem5} and~\ref{lem12},
\begin{eqnarray*}
&& \mathbf{P}\bigl(\tau_x>n, \bigl|S(\nu_n)\bigr|>
\theta_n\sqrt n,\nu_n\leq n^{1-\varepsilon}\bigr)
\\
&&\qquad\le\frac{C}{n^{p/2}} \mathbf E \bigl[\bigl|x+S(\nu_n)\bigr|^p;
\tau_x>\nu_n,\bigl|S(\nu_n)\bigr|>\theta
_n\sqrt n,\nu_n\leq n^{1-\varepsilon} \bigr]
\\
&&\qquad=o\bigl(n^{-p/2}\bigr)=o\bigl(\mathbf P(\tau_x>n)\bigr).
\end{eqnarray*}

Next,
\begin{eqnarray*}
&& \mathbf{P}\bigl(\tau_x>n,x+S(n)\in\sqrt n A, \bigl|S(
\nu_n)\bigr|\le\theta _n\sqrt n,\nu_n\leq
n^{1-\varepsilon}\bigr)
\\
&&\qquad=\sum_{k=1}^{n^{1-\varepsilon}}\int
_{K_{n,\varepsilon}\cap\{
|y-x|\le
\theta_n \sqrt n\}} \mathbf{P}\bigl(\tau_x>k,x+S(k)\in dy,
\nu_n=k\bigr)
\\
&&\hspace*{97pt}\qquad\quad{} \times\mathbf{P}\bigl(\tau_y>n-k,y+S(n-k)\in\sqrt n A\bigr).
\end{eqnarray*}
Using the coupling and arguing as in Lemma~\ref{lem5}, one can show
that
\[
\mathbf{P}\bigl(\tau_y>n-k,y+S(n-k)\in\sqrt n A\bigr)\sim
\mathbf{P}\bigl(\tau^{\mathrm{bm}}_y>n,y+B(n)\in\sqrt n A\bigr)
\]
uniformly in $k\le n^{1-\varepsilon}$ and $y\in K_{n,\varepsilon}$.
Next, we apply asymptotics (\ref{L6.3}) and obtain that
\[
\mathbf{P}\bigl(\tau_y>n-k,y+S(n-k)\in\sqrt n A\bigr) \sim
\varkappa_0\int_A \,dz\, e^{-|z|^2/2}u(y)u(z)
n^{-p/2} %
\]
uniformly in $y\in K_{n,\varepsilon}, |y|\le\theta_n\sqrt n$. As a
result, we obtain
\begin{eqnarray*}
&& \mathbf P\bigl(x+S(n)\in\sqrt n A,\tau_x>n\bigr)\\
&&\qquad\sim \int
_A \,dz\, e^{-|z|^2/2}u(z) n^{-p/2}
\\
&&\qquad\quad{}\times\varkappa_0\mathbf{E}\bigl[u\bigl(x+S(\nu_n)
\bigr),\tau_x>\nu _n,\bigl|S(\nu_n)\bigr|\le
\theta_n\sqrt n,\nu_n\leq n^{1-\varepsilon}\bigr]
\\
&&\qquad\sim\varkappa_0\int_A \,dz
\, e^{-|z|^2/2}u(z) n^{-p/2} V(x),
\end{eqnarray*}
where the latter equivalence holds due to Lemmas~\ref{lem11} and~\ref{lem12}.
Substituting the latter equivalence in (\ref{eq.mmm}) and using the
asymptotics for $\mathbf P(\tau_x>n)$, we arrive at the conclusion.
\end{pf}

Now we change the notation slightly. Let
\[
\mathbf P_x\bigl(S(n)\in A\bigr)=\mathbf P\bigl(x+S(n)\in A\bigr).
\]

\begin{lemma}\label{lem.weak.convergence}
Let $X^n(t)=\frac{S([nt])}{\sqrt n}$ be the family of processes with
the probability measure $\mathbf{\widehat P}^{(V)}_{x\sqrt n},x \in
K$. Then $X^n$ converges weakly in the uniform topology on $D[0,\infty)$
to the Brownian motion conditioned to stay in $K$ with the probability
measure $\mathbf{\widehat P}_x^{(u)}$.
\end{lemma}

\begin{pf}
To prove the
claim, we need to show that the convergence takes place in $D[0,l]$
for every $l$. The proof is identical for each $l$, so we let $l=1$ to
simplify notation. Thus, it is sufficient to show that for every
functional $f\dvtx 0\le f\le1$ uniformly continuous on $D[0,1]$ with respect
to the uniform topology,
\[
\mathbf{\widehat E}_{x\sqrt n}^{(V)} f\bigl(X^n\bigr)
\to\mathbf{\widehat E}_{x}^{(u)} f(B) \qquad\mbox{as } n\to\infty.
\]

We first show that
%
\begin{equation}
\label{tail1} \frac{1}{V(x\sqrt{n})}\mathbf{E} \bigl[V\bigl(x\sqrt {n}+S(n)
\bigr),\bigl|S(n)\bigr|>R\sqrt {n} \bigr]\leq g(R),
\end{equation}
where $g(R)\to0$ as $R\to\infty$.
Using Lemma~\ref{positiv}(a) and (b), we have, for all $R>1$,
\begin{eqnarray*}
&&\frac{1}{V(x\sqrt{n})}\mathbf{E} \bigl[V\bigl(x\sqrt {n}+S(n)\bigr),\bigl|S(n)\bigr|>R\sqrt
{n} \bigr]
\\
&&\qquad\leq\frac{C}{n^{p/2}} \bigl(n^{p/2}\mathbf {P}\bigl(\bigl|S(n)\bigr|>R\sqrt {n}
\bigr)+\mathbf{E} \bigl[\bigl|S(n)\bigr|^p,\bigl|S(n)\bigl|>R\sqrt{n} \bigr] \bigr)
\\
&&\qquad\leq\frac{C}{n^{p/2}}\mathbf{E} \bigl[\bigl|S(n)\bigr|^p,\bigl|S(n)\bigr|>R\sqrt {n}
\bigr].
\end{eqnarray*}
If $p>2$, then
\begin{eqnarray*}
&&\mathbf{E}\bigl[\bigl|S(n)\bigr|^p;\bigl|S(n)\bigr|> R\sqrt{n}\bigr]
\\
&&\qquad= p\int_{R\sqrt{n}}^\infty z^{p-1}\mathbf{P}
\bigl(\bigl|S(n)\bigr|>z\bigr)\,dz +R^pn^{p/2}\mathbf P\bigl(\bigl|S(n)\bigr|>R
\sqrt{n}\bigr). %
\end{eqnarray*}
Choosing $y=z/r$ in the inequality (\ref{NF4}), we have
\[
\mathbf{P}\bigl(\bigl|S(n)\bigr|>z\bigr)\leq C(r) \biggl(\frac{n}{z^2}
\biggr)^r+n\mathbf{P}\bigl(|X|>z/r\bigr). %
\]
Using the latter bound with $r>p/2$, we have
%
\begin{eqnarray}
\label{FN_moment_1}
\mathbf P\bigl(\bigl|S(n)\bigr|>R\sqrt{n}\bigr)&\leq&
C(r)R^{-2r}+n\mathbf{P}\bigl(|X|>R\sqrt {n}/r\bigr)
\nonumber
\\[-8pt]
\\[-8pt]
\nonumber
&\leq& C(r)R^{-2r}+\frac{r^2}{R^2}\mathbf{E}\bigl[|X|^2,|X|>R
\bigr]
\end{eqnarray}
and
%
\begin{eqnarray}
\label{FN_moment}
\nonumber
&&\int_{R\sqrt{n}}^\infty
z^{p-1}\mathbf{P}\bigl(\bigl|S(n)\bigr|>z\bigr)\,dz
\\
&&\qquad\leq C(r) p n^r \int_{R\sqrt{n}}^\infty
z^{p-1-2r}\,dz+np\int_{R\sqrt
{n}}^\infty
z^{p-1}\mathbf{P}\bigl(|X|>z/r\bigr)\,dz
\nonumber
\\[-8pt]
\\[-8pt]
\nonumber
&&\qquad\leq C(r) \frac{p}{2r-p}n^{p/2}R^{p-2r}+r^pn
\mathbf {E}\bigl[|X|^{p},|X|>R\sqrt{n}/r\bigr]
\\
&&\qquad\leq C(p,r)n^{p/2} \bigl(R^{p-2r}+\mathbf{E}
\bigl[|X|^{p},|X|>R\bigr] \bigr)\nonumber
\end{eqnarray}
for all sufficiently large $n$. This implies that (\ref{tail1}) holds
for $p>2$.

If $p\leq2$ then, combining the Markov inequality and (\ref
{FN_moment}), we get
for any $r>1+\delta/2$,
\begin{eqnarray*}
\mathbf{E}\bigl[\bigl|S(n)\bigr|^p;\bigl|S(n)\bigr|> R\sqrt{n}\bigr]&\leq&(R
\sqrt{n})^{p-2-\delta
}\mathbf{ E}\bigl[\bigl|S(n)\bigr|^{2+\delta}, \bigl|S(n)\bigr|> R\sqrt{n}
\bigr]
\\
&\leq& C(2+\delta,r)n^{p/2}R^{2+\delta-2r}.
\end{eqnarray*}
Thus, the bound (\ref{tail1}) is valid for all $p$.

Fix also some $\varepsilon>0$.
It follows easily from Lemma~\ref{positiv}(a), (b)
and the central limit theorem that
\begin{eqnarray*}
&&\frac{1}{V(x\sqrt{n})}\mathbf{E} \bigl[V\bigl(x\sqrt{n}+S(n)\bigr),\tau
_{x\sqrt
{n}}>n,\bigl|S(n)\bigr|\leq R\sqrt{n}, \\
&&\hspace*{107pt}{}\operatorname{ dist}\bigl(x\sqrt{n}+S(n),\partial
K\bigr)\leq\varepsilon\sqrt{n} \bigr]
\\
&&\qquad\leq C\mathbf{P} \bigl(\operatorname{ dist}\bigl(x\sqrt {n}+S(n),\partial K\bigr)\leq
\varepsilon\sqrt{n} \bigr)
\\
&&\qquad\leq C \mathbf{P} \bigl(\operatorname{ dist}\bigl(x+B(1),\partial K\bigr)\leq \varepsilon
\bigr).
\end{eqnarray*}
Since the distribution of $B(1)$ is isotropic,
\[
\mathbf{P} \bigl(\operatorname{ dist}\bigl(x+B(1),\partial K\bigr)\leq\varepsilon \bigr)
\leq C\varepsilon. %
\]
Therefore,
%
\begin{eqnarray}
\label{tail2}
&&\frac{1}{V(x\sqrt{n})}\mathbf{E} \bigl[V\bigl(x
\sqrt{n}+S(n)\bigr),\tau _{x\sqrt
{n}}>n,\bigl|S(n)\bigr|\leq R\sqrt{n},\nonumber\\
&&\hspace*{108pt}{} \operatorname{ dist}
\bigl(x\sqrt{n}+S(n),\partial K\bigr)\leq\varepsilon\sqrt{n} \bigr]
\\
&&\qquad\leq C\varepsilon.\nonumber
\end{eqnarray}
It is clear that similar bounds are valid for the Brownian motion. More
precisely,
%
\begin{equation}
\label{tail3} \frac{1}{u(x)}\mathbf{E} \bigl[u\bigl(x+B(1)\bigr),\bigl|B(1)\bigr|>R
\bigr]\leq g(R)
\end{equation}
and
%
\begin{eqnarray}
\label{tail4}
&&\frac{1}{u(x)}\mathbf{E} \bigl[u\bigl(x+B(1)
\bigr),\tau^{\mathrm{bm}}_x>1,\bigl|B(1)\bigr|\leq R,\operatorname{ dist}\bigl(x+B(1),
\partial K\bigr)\leq\varepsilon \bigr]
\nonumber
\\[-8pt]
\\[-8pt]
\nonumber
&&\qquad \leq C\bigl(|x|+R\bigr)^{d-1}\varepsilon.
\end{eqnarray}
Define
\[
D_n:= \bigl\{\bigl|S(n)\bigr|\leq R\sqrt{n}, \operatorname{ dist}\bigl(x\sqrt{n}+S(n),
\partial K\bigr)\geq\varepsilon\sqrt{n} \bigr\} %
\]
and
\[
D^{\mathrm{bm}}:= \bigl\{\bigl|B(1)\bigr|\leq R, \operatorname{ dist}\bigl(x+B(1),\partial K\bigr)
\geq \varepsilon \bigr\}. %
\]
Using Lemma~\ref{positiv}(a), one can easily get
\begin{eqnarray*}
&&\frac{1}{V(x\sqrt{n})}\mathbf{E} \bigl[f\bigl(X^n\bigr)V\bigl(x
\sqrt{n}+S(n)\bigr)\mathrm{ 1}_{D_n},\tau_{x\sqrt{n}}>n \bigr]
\\
&&\qquad=\bigl(1+o(1)\bigr)\frac{1}{u(x\sqrt{n})}\mathbf{E} \bigl[f\bigl(X^n
\bigr)u\bigl(x\sqrt{n}+S(n)\bigr)\mathrm{1}_{D_n},\tau_{x\sqrt{n}}>n
\bigr]
\\
&&\qquad=\bigl(1+o(1)\bigr)\frac{1}{u(x)}\mathbf{E} \biggl[f\bigl(X^n
\bigr)u \biggl(x+\frac
{S(n)}{\sqrt{n}} \biggr)\mathrm{1}_{D_n},
\tau_{x\sqrt{n}}>n \biggr].
\end{eqnarray*}
We next note that $u(x+\cdot)f(\cdot)\mathrm{1}_{D^{\mathrm{bm}}\cap\{\tau
_x^{\mathrm{bm}}>1\}}$ is bounded and its discontinuities
are a null-set with respect to the Wiener measure on $D[0,1]$ equipped
with the Borel $\sigma$-algebra induced
by the uniform topology. Thus, due to the Donsker invariance principle
on $D[0,1]$ with the uniform topology,
\begin{eqnarray*}
&&\lim_{n\to\infty}\frac{1}{V(x\sqrt{n})}\mathbf{E} \bigl[f
\bigl(X^n\bigr)V\bigl(x\sqrt {n}+S(n)\bigr)\mathrm{1}_{D_n},
\tau_{x\sqrt{n}}>n \bigr]
\\
&&\qquad=\frac{1}{u(x)}\mathbf{E} \bigl[f(B)u\bigl(x+B(1)\bigr)\mathrm{1}_{D^{\mathrm{bm}}},\tau _{x}^{\mathrm{bm}}>1 \bigr].
\end{eqnarray*}
For details on the invariance principle on $D[0,1]$ with the uniform
topology and on the Wiener measure
on this space, we refer to Billingsley's book \cite{Billing}, Section~18.

{F}rom this convergence and bounds (\ref{tail1})--(\ref{tail4}), we
conclude that
\[
\limsup_{n\to\infty}\bigl\llvert \mathbf{\widehat
E}_{x\sqrt n}^{(V)} f\bigl(X^n\bigr)- \mathbf{\widehat
E}_{x}^{(u)} f(B)\bigr\rrvert \leq 2g(R)+C\bigl(|x|+R\bigr)^{d-1}
\varepsilon. %
\]
Letting first $\varepsilon\to0$ and then $R\to\infty$, we get
\[
\limsup_{n\to\infty}\bigl\llvert \mathbf{\widehat
E}_{x\sqrt n}^{(V)} f\bigl(X^n\bigr)- \mathbf{\widehat
E}_{x}^{(u)} f(B)\bigr\rrvert =0. %
\]
Thus, the lemma is proved.\vadjust{\goodbreak}
\end{pf}

\section{Proof of local limit theorems}
\subsection{Preliminary estimates}

\begin{lemma}\label{Loc.L1}
For all $y\in K$ and all $n\geq1$,
%
\begin{equation}
\label{Loc.L1.1} \mathbf{P} \bigl(x+S(n)=y,\tau_x>n \bigr) \leq
\frac{C}{n^{d/2}}\mathbf{P}(\tau_x>n/2) \leq C(x)n^{-p/2-d/2}.
\end{equation}
\end{lemma}

\begin{pf}
It follows easily from (\ref{esseen}) that
%
\begin{equation}
\label{Loc.L1.2} \mathbf{P}\bigl(S(j)=z\bigr)\leq C j^{-d/2},\qquad z\in
\mathbb{Z}^d.
\end{equation}
Therefore, for $m=[n/2]$ we have
\begin{eqnarray*}
&&\mathbf{P}\bigl(x+S(n)=y,\tau_x>n\bigr)
\\
&&\qquad=\sum_{z\in K}\mathbf{P}\bigl(x+S(m)=z,
\tau_x>m\bigr)\mathbf {P}\bigl(z+S(n-m)=y,\tau_z>n-m
\bigr)
\\
&&\qquad\leq\sum_{z\in K}\mathbf{P}\bigl(x+S(m)=z,
\tau_x>m\bigr)\mathbf {P}\bigl(z+S(n-m)=y\bigr)
\\
&&\qquad\leq Cn^{-d/2}\mathbf{P}(\tau_x>m).
\end{eqnarray*}
But we know that $\mathbf{P}(\tau_x>m)\leq C(x)m^{-p/2}$. This
completes the proof of the lemma.
\end{pf}

Comparing (\ref{Loc.L1.1}) with the claim in Theorem~\ref{Loc.T1}, we
see that (\ref{Loc.L1.1}) has the right order for typical values of
$y$, that is, for $y$ of order $n^{1/2}$. But for smaller values of $y$
that bound is too rough.

\begin{lemma}\label{Loc.L2}
For all $x,y\in K$ and all $n\geq1$,
%
\begin{equation}
\label{Loc.L2.1} \mathbf{P}\bigl(x+S(n)=y,\tau_x>n\bigr)\leq
C(x,y)n^{-p-d/2}.
\end{equation}
\end{lemma}

\begin{pf}
We first split the trajectory $S(1),S(2),\ldots,S(n)$ into two parts
\begin{eqnarray*}
&&\mathbf{P}\bigl(x+S(n)=y,\tau_x>n\bigr)
\\
&&\qquad=\sum_{z\in
K}\mathbf{P}\bigl(x+S(m)=z,
\tau_x>m\bigr)\mathbf{P}\bigl(z+S(n-m)=y,\tau_z>n-m
\bigr),
\end{eqnarray*}
where $m=[n/2]$. Then we reverse the time in the second part:
\begin{eqnarray*}
&&\mathbf{P} \bigl(z+S(n-m)=y,\tau_z>n-m \bigr)
\\
&&\qquad=\mathbf{P} \bigl(z+S(k)\in K, k=1,2,\ldots,n-m-1,z+S(n-m)=y \bigr)
\\
&&\qquad=\mathbf{P} \Biggl(z+S(n-m)-\sum_{j=k+1}^{n-m}
X(j)\in K, k=1,2,\ldots,n-m-1,\\
&&\hspace*{227pt}{}z+S(n-m)=y \Biggr)
\\
&&\qquad=\mathbf{P} \bigl(y-S(k)\in K, k=1,2,\ldots,n-m-1,y-S(n-m)=z \bigr)
\\
&&\qquad=\mathbf{P} \bigl(y-S(n-m)=z,\tau'_y>n-m \bigr),
\end{eqnarray*}
where $\tau'_y=\min\{k\geq1\dvtx y-S(k)\notin K\}$. Applying Lemma~\ref{Loc.L1} to the random walk $\{-S(n)\}$, we obtain
\[
\mathbf{P}\bigl(z+S(n-m)=y,\tau_z>n-m\bigr)\leq
C(y)n^{-p/2-d/2}. %
\]
Consequently,
\begin{eqnarray*}
\mathbf{P}\bigl(x+S(n)=y,\tau_x>n\bigr)&\leq& C(y)n^{-p/2-d/2}
\sum_{z\in
K}\mathbf {P}\bigl(x+S(m)=z,
\tau_x>m\bigr)
\\
&\leq& C(y)n^{-p/2-d/2}C(x)n^{-p/2}.
\end{eqnarray*}
Thus, the proof is finished.
\end{pf}

\begin{lemma}\label{Loc.L3}
There exist constants $a$ and $C$ such that, for every $u>0$,
%
\begin{equation}
\label{Loc.L3.1} \limsup_{n\to\infty}\sup_{|x-z|\geq u\sqrt{n}}n^{d/2}
\mathbf {P} \bigl(x+S(n)=z \bigr)\leq C\exp\bigl\{-au^2\bigr\}
\end{equation}
and
%
\begin{equation}
\label{Loc.L3.2} \limsup_{n\to\infty}\sup_{x,z\in M_{n,u}}n^{d/2}
\mathbf{P} \bigl(x+S(n)=z,\tau_x\leq n \bigr)\leq C\exp\bigl
\{-au^2\bigr\},
\end{equation}
where $M_{n,u}:= \{z\dvtx \operatorname{ dist}(z,\partial K)\geq u\sqrt{n}
\}$.
\end{lemma}

\begin{pf}
Put again $m=[n/2]$.
For $x$ and $z$ with $|x-z|\geq u\sqrt{n}$, we have
\begin{eqnarray*}
\mathbf{P} \bigl(x+S(n)=z \bigr)&\leq &\mathbf{P} \bigl(x+S(n)=z,\bigl|S(m)\bigr|\geq u
\sqrt{n}/2 \bigr)
\\
&&{}+\mathbf{P} \bigl(x+S(n)=z,\bigl|S(n)-S(m)\bigr|\geq u\sqrt{n}/2 \bigr).
\end{eqnarray*}

We first note that from the Markov property and (\ref{Loc.L1.2}) follows
\[
\mathbf{P} \bigl(x+S(n)=z,\bigl|S(m)\bigr|\geq u\sqrt{n}/2 \bigr)\leq Cn^{-d/2}
\mathbf{P} \bigl(\bigl|S(m\bigr)\bigr|\geq u\sqrt{n}/2 \bigr). %
\]
Reversing the time, as it was done in the previous lemma, we infer
that
\[
\mathbf{P} \bigl(x+S(n)=z,\bigl|S(n)-S(m)\bigr|\geq u\sqrt{n}/2 \bigr)\leq
Cn^{-d/2} \mathbf{P} \bigl(\bigl|S(n-m)\bigr|\geq u\sqrt{n}/2 \bigr). %
\]
As a result we have
\begin{eqnarray*}
&\mathbf{P} \bigl(x+S(n)=z \bigr)\leq Cn^{-d/2} \bigl(\mathbf{P}
\bigl(\bigl|S(m)\bigr|\geq u\sqrt{n}/2 \bigr) +\mathbf{P} \bigl(\bigl|S(n-m)\bigr|\geq u\sqrt{n}/2
\bigr) \bigr).
\end{eqnarray*}
The first estimate in the lemma follows now from the central limit theorem.

To prove the second estimate we note that if $\operatorname{ dist}(z,\partial
K)\geq u\sqrt{n}$ then,
using the Markov property, we obtain
\[
\mathbf{P} \bigl(x+S(n)=z,\tau_x\leq n/2 \bigr)\leq \max
_{n/2\leq k\leq n}\sup_{|y-z|\geq u\sqrt{n}}n^{d/2}\mathbf {P}
\bigl(y+S(k)=z \bigr). %
\]
Furthermore, if $\operatorname{ dist}(x,\partial K)\geq u\sqrt{n}$ then,
reversing additionally the time,
we get
\[
\mathbf{P} \bigl(x+S(n)=z,n/2<\tau_x\leq n \bigr)\leq \max
_{n/2\leq k\leq n}\sup_{|y-x|\geq u\sqrt{n}}n^{d/2}\mathbf {P}
\bigl(y+S(k)=x \bigr). %
\]
Applying (\ref{Loc.L3.1}), we complete the proof.
\end{pf}

\subsection{Proof of Theorem \texorpdfstring{\protect\ref{Loc.T1}}{5}}
For simplicity we assume that $X$ takes values on $\mathbb{Z}^d$.

We split the cone into three parts:
\begin{eqnarray*}
K^{(1)}&:=&\bigl\{y\in K\dvtx |y|>A\sqrt{n}\bigr\},
\\
K^{(2)}&:=&\bigl\{y\in K\dvtx |y|\leq A\sqrt{n},\operatorname{ dist}(y,\partial
K)\leq 2\varepsilon\sqrt{n}\bigr\},
\\
K^{(3)}&:=&\bigl\{y\in K\dvtx |y|\leq A\sqrt{n},\operatorname{ dist}(y,\partial
K)> 2\varepsilon\sqrt{n}\bigr\}
\end{eqnarray*}
with some $A>0$ and $\varepsilon>0$. Noting that
\[
\lim_{A\to\infty}\sup_{y\in K^{(1)}}u(y/\sqrt{n})
e^{-|y|^2/2n}=0 %
\]
and
\[
\lim_{A\to\infty}\lim_{\varepsilon\to0}\sup
_{y\in
K^{(2)}}u(y/\sqrt{n}) e^{-|y|^2/2n}=0, %
\]
one can easily see that the theorem will be proved if we show that
\begin{eqnarray*}
\lim_{A\to\infty}\limsup_{n\to\infty}n^{p/2+d/2}
\sup_{y\in
K^{(1)}}\mathbf{P}\bigl(x+S(n)=y,\tau_x>n
\bigr)&=&0, %
\\
\lim_{A\to\infty}\lim_{\varepsilon\to0}\limsup
_{n\to\infty
}n^{p/2+d/2}\sup_{y\in K^{(2)}}\mathbf{P}
\bigl(x+S(n)=y,\tau_x>n\bigr)&=&0 %
\end{eqnarray*}
and
\begin{eqnarray*}
&&\lim_{\varepsilon\to0}\limsup_{n\to\infty}\sup
_{y\in K^{(3)}} \biggl|n^{p/2+d/2}\mathbf{P} \bigl(x+S(n)=y,
\tau_x>n \bigr)
\\
&&\hspace*{96pt}{}-\varkappa V(x)H_0u \biggl(\frac{y}{\sqrt{n}}
\biggr)e^{-|y|^2/2n} \biggr|=0.
\end{eqnarray*}
This is done in (\ref{Loc.0}), (\ref{Loc.00}) and (\ref{Loc.6}), respectively.

We have
\begin{eqnarray*}
\mathbf{P}\bigl(x+S(n)=y,\tau_x>n\bigr)&=& \mathbf{P}\bigl(x+S(n)=y,
\tau_x>n,\bigl|S(n/2)\bigr|\leq A\sqrt{n}/2\bigr)
\\
&&{}+\mathbf{P}\bigl(x+S(n)=y,\tau_x>n,\bigl|S(n/2)\bigr|> A\sqrt{n}/2\bigr).
\end{eqnarray*}
Using the Markov property and (\ref{Loc.L1.2}), we get, for all $y\in K^{(1)}$,
\begin{eqnarray*}
&&\mathbf{P}\bigl(x+S(n)=y,\tau_x>n,\bigl|S(n/2)\bigr|> A\sqrt{n}/2\bigr)
\\
&&\qquad\leq C(x)n^{-d/2-p/2}\mathbf{P} \bigl(\bigl|x+S(n/2)\bigr|>A\sqrt
{n}/2-|x||
\tau_x>n/2 \bigr).
\end{eqnarray*}
Applying now (\ref{eq.cond.dist}) in Theorem~\ref{T2}, we obtain,
uniformly in $y\in K^{(1)}$,
\begin{eqnarray*}
&&\lim_{A\to\infty}\limsup_{n\to\infty} n^{p/2+d/2}
\mathbf {P}\bigl(x+S(n)=y,\tau_x>n,\bigl|S(n/2)\bigr|> A\sqrt{n}/2\bigr)
\\
&&\qquad\leq C(x)\lim_{A\to\infty} \mu\bigl(\bigl\{z\in K\dvtx |z|>A/\sqrt {2}\bigr\}
\bigr)=0.
\end{eqnarray*}
Furthermore, applying Theorem~\ref{T} and (\ref{Loc.L3.1}), we get, for
$|y|>A\sqrt{n}$,
\begin{eqnarray*}
&&\mathbf{P}\bigl(x+S(n)=y,\tau_x>n,\bigl|S(n/2)\bigr|\leq A\sqrt{n}/2\bigr)
\\
&&\qquad\leq \mathbf{P}(\tau_x>n/2)\sup_{|y-z|>A\sqrt{n}/2}\mathbf
{P}\bigl(x+z+S(n/2)=y\bigr)
\\
&&\qquad\leq C(x)n^{-d/2-p/2}\exp\bigl\{-aA^2/4\bigr\}.
\end{eqnarray*}
As a result, we have
%
\begin{equation}
\label{Loc.0} \lim_{A\to\infty}\limsup_{n\to\infty}n^{d/2+p/2}
\sup_{y\in
K^{(1)}}\mathbf{P}\bigl(x+S(n)=y,\tau_x>n
\bigr)=0.
\end{equation}

We next consider $y\in K^{(2)}$.
Set $m=[n/2]$. Using the time reversion from Lemma~\ref{Loc.L2} and the
bound (\ref{Loc.L1.1}), we obtain
\begin{eqnarray*}
&&\mathbf{P}\bigl(x+S(n)=y,\tau_x>n\bigr)
\\
&&\qquad=\sum_{z\in K}\mathbf{P}\bigl(x+S(m)=z,
\tau_x>m\bigr)\mathbf {P}\bigl(y-S(n-m)=z,\tau '_y>n-m
\bigr)
\\
&&\qquad\leq C(x)m^{-p/2-d/2}\sum_{z\in K}\mathbf{P}
\bigl(y-S(n-m)=z,\tau '_y>n-m\bigr)
\\
&&\qquad\leq C(x)n^{-p/2-d/2}\mathbf{P}\bigl(\tau'_y>n-m
\bigr).
\end{eqnarray*}
We want to show that
%
\begin{equation}
\label{un.bound} \limsup_{n\to\infty}\sup_{y\in K^{(2)}}
\mathbf{P}\bigl(\tau '_y>n-m\bigr)\leq g(\varepsilon)
\end{equation}
with some $g(\varepsilon)\to0$ as $\varepsilon\to0$.
Using the same arguments as in (\ref{L5.1}), we have
\[
\mathbf{P}\bigl(\tau'_y>n-m\bigr)\leq\mathbf{P}\bigl(
\tau^{\mathrm{bm}}_{y+\varepsilon
\sqrt
{n}x_0}>n-m\bigr)+o\bigl(n^{-r}\bigr),
\]
and $o(n^{-r})$ is uniform in $y$. Consequently, by the scaling
property of the Brownian motion,
%
\begin{equation}
\label{newbound} \sup_{y\in K^{(2)}}\mathbf{P}\bigl(
\tau'_y>n-m\bigr)\leq \sup_{z\in K\dvtx |z|\leq A,\operatorname{ dist}(z,\partial K)\leq2\varepsilon
}
\mathbf{P}\bigl(\tau^{\mathrm{bm}}_{z+\varepsilon x_0}>1/2\bigr)+o
\bigl(n^{-r}\bigr).\hspace*{-17pt}
\end{equation}

Note that if $\operatorname{ dist}(z,\partial K)\leq2\varepsilon$ then
$\operatorname{ dist}(z+\varepsilon x_0,\partial K)\leq C_*\varepsilon$.

The most standard way of bounding $\mathbf{P}(\tau^{\mathrm{bm}}_{x}>1/2)$ is
the use of
the parabolic boundary Harnack principle which gives
%
\begin{equation}
\label{harnack} \mathbf{P}\bigl(\tau^{\mathrm{bm}}_{x}>1/2\bigr)
\leq Cu(x),
\end{equation}
see \cite{Var99}, page 336, and references there. If $|x|$ is bounded
and $\operatorname{ dist}(x,\partial K)\leq C_*\varepsilon$, then
(\ref{un.bound}) is immediate from the definition of $u$.

But for convex cones there exists an elementary way of deriving (\ref
{un.bound}) from (\ref{newbound}), which we present below.

If $K$ is convex, then there exists a hyperplane $H=H(z)$ such that
$\operatorname{ dist}(z+\varepsilon x_0,H)\leq2C_*\varepsilon$ and $K\cap
H=\varnothing$. If we set
$T_z:=\inf\{t>0\dvtx z+B(t)\in H\}$ then, obviously,
\[
\mathbf{P}\bigl(\tau^{\mathrm{bm}}_{z+\varepsilon x_0}>1/2\bigr) \leq
\mathbf{P}(T_{z+\varepsilon x_0}>1/2). %
\]
Due to the rotational invariance of the Brownian motion, the normal to
$H$ component of $B$ is a one-dimensional
Brownian motion. As a result, we have
\[
\mathbf{P}(T_{z+\varepsilon x_0}>1/2)\leq \mathbf{P} \Bigl(2C_*\varepsilon+\inf
_{t\leq1/2}B_1(t)>0 \Bigr) %
\]
uniformly in $z$ satisfying $\operatorname{ dist}(z,\partial K)\leq2\varepsilon$.
Applying finally the reflection principle,
we conclude from (\ref{newbound}) that
\[
\mathbf{P}\bigl(\tau'_y>n-m\bigr)\leq C \varepsilon+
o\bigl(n^{-r}\bigr) %
\]
uniformly in $y$ satisfying $\operatorname{ dist}(y,\partial K)\leq2\varepsilon
\sqrt{n}$.

Summarizing,
%
\begin{equation}
\label{Loc.00} \lim_{A\to\infty}\lim_{\varepsilon\to0}
\limsup_{n\to\infty}n^{d/2+p/2} \sup_{y\in K^{(2)}}
\mathbf{P}\bigl(x+S(n)=y,\tau_x>n\bigr)=0.
\end{equation}

It remains to consider ``typical'' values of $y$, that is, $y\in
K^{(3)}$. Set
$m=[\varepsilon^3 n]$. We start with the representation
%
\begin{eqnarray}
\label{Loc.1}
&&\mathbf{P}\bigl(x+S(n)=y,\tau_x>n\bigr)
\nonumber
\\[-8pt]
\\[-8pt]
\nonumber
&&\qquad=\sum_{z\in
K}\mathbf{P}\bigl(x+S(n-m)=z,
\tau_x>n-m\bigr)\mathbf{P}\bigl(z+S(m)=y,\tau_z>m
\bigr).
\end{eqnarray}

Let $K_1(y):=\{z\in K\dvtx |z-y|<\varepsilon\sqrt{n}\}$. Applying
(\ref{Loc.L3.1}), we have
%
\begin{eqnarray}
\label{Loc.2}
\nonumber
&&\sum_{z\in K\setminus K_1(y)}\mathbf{P}
\bigl(x+S(n-m)=z,\tau _x>n-m\bigr)\mathbf {P}\bigl(z+S(m)=y,
\tau_z>m\bigr)
\\
\nonumber
&&\qquad\leq\sum_{z\in K\setminus K_1(y)}\mathbf {P}
\bigl(x+S(n-m)=z,\tau_x>n-m\bigr) \mathbf{P}\bigl(z+S(m)=y\bigr)
\\
&&\qquad\leq\sum_{z\in K\setminus K_1(y)}\mathbf {P}
\bigl(x+S(n-m)=z,\tau_x>n-m\bigr) Cn^{-d/2}
\varepsilon^{-3d/2}\exp\{ -a/\varepsilon\}
\\
\nonumber
&&\qquad\leq C\mathbf{P}(\tau_x>n-m)n^{-d/2}\varepsilon
^{-3d/2}\exp\{-a/\varepsilon\}
\\
&&\qquad\leq CV(x)n^{-d/2-p/2}\varepsilon^{-3d/2}\exp\{-a/\varepsilon\}\nonumber
\end{eqnarray}
uniformly in $y$ satisfying $\operatorname{ dist}(y,\partial K)>2\varepsilon
\sqrt{n}$.

If $\operatorname{ dist}(y,\partial K)>2\varepsilon\sqrt{n}$ and $z\in K_1(y)$, then
$\operatorname{ dist}(z,\partial K)>\varepsilon\sqrt{n}$. Using (\ref{Loc.L3.2}),
we have
%
\begin{eqnarray}
\label{Loc.3}
\nonumber
&&\sum_{z\in K_1(y)}\mathbf{P}
\bigl(x+S(n-m)=z,\tau_x>n-m\bigr)\mathbf {P}\bigl(z+S(m)=y,
\tau_z\leq m\bigr)
\\
&&\qquad\leq\sum_{z\in K_1(y)}\mathbf{P}
\bigl(x+S(n-m)=z,\tau_x>n-m\bigr) Cn^{-d/2}
\varepsilon^{-3d/2}\exp\{-a/\varepsilon\}
\nonumber
\\[-8pt]
\\[-8pt]
\nonumber
&&\qquad\leq C\mathbf{P}(\tau_x>n-m)n^{-d/2}\varepsilon
^{-3d/2}\exp\{-a/\varepsilon\}
\\
&&\qquad\leq CV(x)n^{-d/2-p/2}\varepsilon^{-3d/2}\exp\{-a/\varepsilon\}\nonumber
\end{eqnarray}
uniformly in $y$ satisfying $\operatorname{ dist}(y,\partial K)>2\varepsilon
\sqrt{n}$.

Using the local limit theorem for unconditioned random walks (see
Proposition~7.9 in Spitzer's book \cite{Sp76}), we
have, uniformly in $y$,
%
\begin{eqnarray}
\label{Loc.4}
\nonumber
\Sigma(y)&:=&\sum_{z\in K_1(y)}
\mathbf{P}\bigl(x+S(n-m)=z,\tau _x>n-m\bigr)\mathbf {P}
\bigl(z+S(m)=y\bigr)
\\
&=&\sum_{z\in K_1(y)}\mathbf{P}\bigl(x+S(n-m)=z,
\tau_x>n-m\bigr) \bigl(2\pi n\varepsilon ^3
\bigr)^{-d/2}
\nonumber
\\[-8pt]
\\[-8pt]
\nonumber
&&\hspace*{29pt}{}\times\exp\bigl\{-|y-z|^2/2\varepsilon^3n
\bigr\}
\\
\nonumber
&&{}+O \bigl(n^{-d/2-p/2}\varepsilon ^{-3d/2}e^{-a/\varepsilon}
\bigr).
\end{eqnarray}
It follows from compactness argument and the integral limit theorem for
$\{S(n)\}$ conditioned to
stay in $K$ that
\begin{eqnarray*}
&&\limsup_{n\to\infty}\sup_{y\in K^{(3)}}\biggl |\sum
_{z\in
K_1(y)}\mathbf {P}\bigl(x+S(n-m)=z|\tau_x>n-m
\bigr) \exp\bigl\{-|y-z|^2/2\varepsilon^3n\bigr\}
\\
&&\hspace*{37pt}\qquad{}-H_0\int_{\llvert (1-\varepsilon^3)^{1/2}r-
{y}/{\sqrt
{n}}\rrvert <\varepsilon} u(r)e^{-|r|^2/2}e^{-|(1-\varepsilon^3)^{1/2}r-y/\sqrt
{n}|^2/2\varepsilon^3}\,dr
\biggr|\\
&&\qquad=0
\end{eqnarray*}
for every fixed $\varepsilon$.
Set, for brevity,
\[
I_1(y,n,\varepsilon):=\int_{\llvert (1-\varepsilon^3)^{1/2}r-
{y}/{\sqrt{n}}\rrvert <\varepsilon}
u(r)e^{-|r|^2/2}e^{-|(1-\varepsilon^3)^{1/2}r-y/\sqrt
{n}|^2/2\varepsilon^3}\,dr %
\]
and
\[
I_2(y,n,\varepsilon):=\int_{\llvert (1-\varepsilon^3)^{1/2}r-
{y}/{\sqrt{n}}\rrvert <\varepsilon}
e^{-|(1-\varepsilon^3)^{1/2}r-y/\sqrt{n}|^2/2\varepsilon^3}\,dr. %
\]
Since $u(r)e^{-|r|^2/2}$ is uniformly continuous, we have
\begin{eqnarray*}
\limsup_{\varepsilon\to0}\sup_{n\geq1}\sup
_{y\in K^{(3)}} \frac{|I_1(y,n,\varepsilon)-u(y/\sqrt {n})e^{-|y|^2/2n}I_2(y,n,\varepsilon)|}{I_2(y,n,\varepsilon)}=0.
\end{eqnarray*}
Noting that
\begin{eqnarray*}
I_2(n,y,\varepsilon) &=&\bigl(1-\varepsilon^3
\bigr)^{-d/2}\varepsilon^{3d/2}\int_{|r'|<\varepsilon
^{-1/2}}
e^{-|r'|^2/2}\,dr'
\\
&\sim&\varepsilon^{3d/2}\int_{\mathbb{R}^d}
e^{-|r'|^2/2}\,dr' =\bigl(2\pi\varepsilon^3
\bigr)^{d/2},
\end{eqnarray*}
we conclude that
\begin{eqnarray*}
\limsup_{\varepsilon\to0}\sup_{n\geq1}\sup
_{y\in K^{(3)}} \frac{|I_1(y,n,\varepsilon)-u(y/\sqrt{n})e^{-|y|^2/2n}(2\pi
\varepsilon
^3)^{d/2}|}{(2\pi\varepsilon^3)^{d/2}}=0.
\end{eqnarray*}
Consequently,
\begin{eqnarray*}
&&\limsup_{n\to\infty}\sup_{y\in K^{(3)}}\biggl |\sum
_{z\in
K_1(y)}\mathbf {P}\bigl(x+S(n-m)=z|\tau_x>n-m
\bigr) \exp\bigl\{-|y-z|^2/2\varepsilon^3n\bigr\}
\\
&&\hspace*{165pt}\qquad{}-H_0u(y/\sqrt{n})e^{-|y|^2/2n}\bigl(2\pi\varepsilon
^3\bigr)^{d/2} \biggr|\\
&&\qquad=o\bigl(\varepsilon^{3d/2}\bigr).
\end{eqnarray*}

{F}rom this relation and (\ref{Loc.4}), we infer
%
\begin{equation}
\label{Loc.5} \lim_{\varepsilon\to0}\limsup_{n\to\infty}
\sup_{y\in K^{(3)}} \bigl|n^{d/2+p/2}\Sigma(y)-\varkappa
V(x)H_0u(y/\sqrt {n})e^{-|y|^2/2n} \bigr|=0.
\end{equation}
Combining (\ref{Loc.1}), (\ref{Loc.2}), (\ref{Loc.3}) and
(\ref{Loc.5}), we obtain
%
\begin{eqnarray}
\label{Loc.6}
&&\lim_{\varepsilon\to0}\limsup
_{n\to\infty}\sup_{y\in K^{(3)}} \biggl|n^{p/2+d/2}\mathbf{P}
\bigl(x+S(n)=y,\tau_x>n \bigr)
\nonumber
\\[-8pt]
\\[-8pt]
\nonumber
&&\hspace*{95pt}{}-\varkappa V(x)H_0u \biggl(\frac{y}{\sqrt{n}}
\biggr)e^{-|y|^2/2n}\biggr |=0.
\end{eqnarray}
%
\subsection{Proof of Theorem \texorpdfstring{\protect\ref{Loc.T2}}{6}} Set $m=[(1-t)n]$
and write
%
\begin{eqnarray}
\label{Loc.7}
\nonumber
&&\mathbf{P}\bigl(x+S(n)=y,\tau_x>n\bigr)
\\
&&\qquad=\sum_{z\in K}\mathbf{P}\bigl(x+S(n-m)=z,
\tau_x>n-m\bigr)\mathbf {P}\bigl(z+S(m)=y,\tau_z>m
\bigr)
\\
&&\qquad=\sum_{z\in
K}\mathbf{P}\bigl(x+S(n-m)=z,
\tau_x>n-m\bigr)\mathbf{P}\bigl(y+S'(m)=z,
\tau'_y>m\bigr),\nonumber
\end{eqnarray}
where $S'$ is distributed as $-S$.

We first note that, according to Theorem~\ref{T} and Lemma~\ref{Loc.L1},
\begin{eqnarray*}
\Sigma_1(A,n)&:=&\sum_{z\in K:|z|>A\sqrt{n}}\mathbf {P}
\bigl(x+S(n-m)=z,\tau_x>n-m\bigr)\\
&&\hspace*{51pt}{}\times \mathbf{P}\bigl(y+S'(m)=z,
\tau'_y>m\bigr)
\\
&\leq& C(x,y)n^{-p-d/2}\mathbf{P}\bigl(\bigl|S'(m)\bigr|>A\sqrt{n}-|y| |
\tau'_y>m\bigr).
\end{eqnarray*}
Therefore, in view of Theorem~\ref{T2},
%
\begin{equation}
\label{Loc.8} \lim_{A\to\infty}\lim_{n\to\infty}n^{p+d/2}
\Sigma_1(A,n)=0.
\end{equation}
Applying (\ref{Loc.T1.1}) first to $\{S'(n)\}$ and then to $\{S(n)\}$,
we get
\begin{eqnarray*}
\Sigma_2(A,n)&:=&\sum_{z\in K:|z|\leq A\sqrt{n}}\mathbf {P}
\bigl(x+S(n-m)=z,\tau_x>n-m\bigr)\\
&&\hspace*{50pt}{}\times\mathbf{P}\bigl(y+S'(m)=z,
\tau'_y>m\bigr)
\\
&=&\varkappa^2\frac{V(x)V'(y)H_0^2}{ (t(1-t)
)^{p/2+d/2}}n^{-p-d}
\\
&&{}\times\sum_{z\in K:|z|\leq
A\sqrt{n}}u \biggl(\frac{z}{\sqrt{tn}}
\biggr) u \biggl(\frac{z}{\sqrt{(1-t)n}} \biggr)\\
&&\hspace*{65pt}{}\times \exp \biggl\{-\frac
{|z|^2}{2tn}-
\frac{|z|^2}{2(1-t)n} \biggr\}\\
&&{}+o(R_n),
\end{eqnarray*}
where
\begin{eqnarray*}
R_n&=&\mathbf{P}(\tau_x>n-m)m^{-p/2-d/2}\\
&&{} +
\bigl(n(n-m)\bigr)^{-p/2-d/2}\sum_{|z|\leq A\sqrt{n}}u \biggl(
\frac{z}{\sqrt {tn}} \biggr) \exp \biggl\{-\frac{|z|^2}{2tn} \biggr\}. %
\end{eqnarray*}
Using Theorem~\ref{T} and noting that the sum is of order $n^{d/2}$, we
conclude that
$R_n\leq Cn^{-p-d/2}$. Therefore,
\begin{eqnarray*}
\Sigma_2(A,n)&=&\varkappa^2\frac{V(x)V'(y)H_0^2}{ (t(1-t)
)^{p/2+d/2}}n^{-p-d}
\\
&&{}\times\sum_{z\in K:|z|\leq
A\sqrt{n}}u \biggl(\frac{z}{\sqrt{tn}}
\biggr) u \biggl(\frac{z}{\sqrt{(1-t)n}} \biggr)\exp \biggl\{-\frac
{|z|^2}{2tn}-
\frac{|z|^2}{2(1-t)n} \biggr\}
\\
&&{}+o\bigl(n^{-p-d/2}\bigr).
\end{eqnarray*}

Thus, it remains to compute the limiting value of the sum in the latter formula.
Using the homogeneity of $u$, we get
%
\begin{eqnarray}
\label{Loc.10}
&&
\lim_{n\to\infty} n^{-d/2}\sum
_{z\in K:|z|\leq
A\sqrt{n}}u \biggl(\frac{z}{\sqrt{tn}} \biggr) u \biggl(
\frac{z}{\sqrt{(1-t)n}} \biggr)\exp \biggl\{-\frac
{|z|^2}{2t(1-t)n} \biggr\}
\nonumber
\\[-8pt]
\\[-8pt]
\nonumber
&&\qquad=\frac{1}{ (t(1-t) )^{p/2}} \int_{w\in K:|w|\leq
A}u^2(w)e^{-|w|^2/2t(1-t)}\,dw.
\end{eqnarray}
Consequently,
%
\begin{eqnarray}
\label{Loc.9}
&&\lim_{A\to\infty}\lim
_{n\to\infty}n^{p+d/2}\Sigma_2(A,n)
\nonumber
\\[-8pt]
\\[-8pt]
\nonumber
&&\qquad =\varkappa^2\frac{V(x)V'(y)H_0^2}{ (t(1-t) )^{p+d/2}} \int_{K}u^2(w)e^{-|w|^2/2t(1-t)}\,dw.
\end{eqnarray}
Combining (\ref{Loc.7}), (\ref{Loc.8}) and (\ref{Loc.9}), we obtain
%
\begin{eqnarray}
\label{Loc.11}
&&\lim_{n\to\infty}n^{p+d/2}
\mathbf{P}\bigl(x+S(n)=y,\tau_x>n\bigr)
\nonumber
\\[-8pt]
\\[-8pt]
\nonumber
&&\qquad =\varkappa^2\frac{V(x)V'(y)H_0^2}{ (t(1-t) )^{p+d/2}} \int_{K}u^2(w)e^{-|w|^2/2t(1-t)}\,dw.
\end{eqnarray}
Substituting $v=w/\sqrt{t(1-t)}$ we see that (\ref{Loc.T1.2}) holds with
$\rho=\varkappa^2\int_{K}u^2(v)\times  e^{-|v|^2/2}\,dv$.

Repeating the derivation of (\ref{Loc.10}), we obtain
\begin{eqnarray*}
&&\lim_{n\to\infty} n^{p+d/2}\mathbf{P} \biggl(
\frac{x+S([tn])}{\sqrt{n}}\in D,x+S(n)=y,\tau _x>n \biggr)
\\
&&\qquad=\frac{V(x)V'(y)H_0^2}{(2\pi)^{d} (t(1-t) )^{p+d/2}} \int_{D}u^2(w)e^{-|w|^2/2t(1-t)}\,dw.
\end{eqnarray*}
Combining this with (\ref{Loc.11}), we get (\ref{Loc.T1.3}). Thus, the
proof is finished.

\section*{Acknowledgements}
The authors are grateful to the referees for a number of comments and
suggestions that helped us to improve the exposition.

%



\printaddresses

\end{document}